\documentclass[journal]{IEEEtran}

\usepackage{amsmath,amsfonts}
\allowdisplaybreaks

\usepackage{algorithm}
\usepackage{algorithmic}

\usepackage{array}
\usepackage[caption=false,font=normalsize,labelfont=sf,textfont=sf]{subfig}
\usepackage{textcomp}
\usepackage{stfloats}
\usepackage{url}
\usepackage{verbatim}
\usepackage{graphicx}
\usepackage{cite}
\RequirePackage{amsthm,amssymb,mathrsfs}
\RequirePackage[numbers]{natbib}
\usepackage{hyperref}
\usepackage{enumitem}  
\usepackage{color}
\usepackage{bm,bbm}

\newtheorem{theorem}{Theorem}
\newtheorem{corollary}{Corollary}
\newtheorem{lemma}{Lemma}
\newtheorem{definition}{Definition}
\newtheorem{example}{Example}
\theoremstyle{remark}
\newtheorem{remark}{Remark}

\DeclareMathOperator*{\essup}{ess\,sup}
\newcommand{\Real}{{\mathbb{R}}}
\newcommand{\Exp}{\mathbf{E}}
\newcommand{\Pro}{\mathbf{P}}

\newcommand{\cF}{\mathscr{F}}

\newcommand{\1}{\mathbbm{1}}
\newcommand{\ind}[1]{\1_{\{#1\}}}

\newcommand{\algname}{$\delta$-DP-CUSUM}

\begin{document}

\title{Sequential Change Detection with \\ Differential Privacy}

\author{Liyan~Xie and Ruizhi~Zhang
\thanks{Liyan Xie (Email:
liyanxie@umn.edu) is with Department of Industrial and Systems Engineering, University of Minnesota, Minneapolis, MN 55455 USA. Ruizhi Zhang (Email: ruizhi.zhang@uga.edu) is with the Department of Statistics, University of Georgia, Athens, GA 30602 USA.}
\thanks{Manuscript received July 18, 2025; revised October 21, 2025;
accepted December 9, 2025. (Corresponding author: Ruizhi Zhang.)}

}

\markboth{IEEE Transactions on Information Theory}%
{Shell \MakeLowercase{\textit{et al.}}: Bare Demo of IEEEtran.cls for IEEE Journals}


\maketitle

\begin{abstract}
Sequential change detection is a fundamental problem in statistics and signal processing, with the CUSUM procedure widely used to achieve minimax detection delay under a prescribed false-alarm rate when pre- and post-change distributions are fully known. However, releasing CUSUM statistics and the corresponding stopping time directly can compromise individual data privacy. We therefore introduce a differentially private (DP) variant, called DP-CUSUM, that injects calibrated Laplace noise into both the vanilla CUSUM statistics and the detection threshold, preserving the recursive simplicity of the classical CUSUM statistics while ensuring per-sample differential privacy. We derive closed-form bounds on the average run length to false alarm and on the worst-case average detection delay, explicitly characterizing the trade-off among privacy level, false-alarm rate, and detection efficiency. Our theoretical results imply that under a weak privacy constraint, our proposed DP-CUSUM procedure achieves the same first-order asymptotic optimality as the classical, non-private CUSUM procedure. Numerical simulations are conducted to demonstrate the detection efficiency of our proposed DP-CUSUM under different privacy constraints, and the results are consistent with our theoretical findings.
\end{abstract}


\section{Introduction}

Sequential change detection is a fundamental problem in statistics and signal processing, with applications across a wide range of practical domains. The canonical formulation considers a sequence of observations sampled independently, with an unknown changepoint at which the underlying distribution switches from one distribution to an alternative. The goal of sequential change detection is to detect the occurrence of distributional change with minimal delay while controlling the false-alarm rate. This problem is of major importance in many applications, such as seismic event detection \cite{xie2019asynchronous}, quality control \cite{shi2009quality}, dynamical systems \cite{lai1995sequential}, healthcare \cite{balageas2010structural}, social networks \cite{li2017detecting}, anomaly detection \cite{chandola2009anomaly}, detection of attacks \cite{tartakovsky2014rapid}, etc. However, in many of these applications, the data may contain sensitive personal information
such as financial or medical records \cite{cai2021cost}, while traditional sequential detection procedures typically release the decision directly without appropriate privacy protection. Thus, procedures with
good detection ability while preserving individuals’ information are highly desirable. 

Since Dwork \emph{et al}.’s pioneering work \cite{dwork2006calibrating}, differential privacy has garnered much attention. A wide variety of differentially private procedures with theoretical efficiency guarantees have been developed for many statistical problems, such as point estimation \cite{cai2021cost,avella2020privacy} and hypothesis testing \cite{degue2018differentially,canonne2019structure,zhang2024detection}.
Informally, differential privacy provides a systematic framework for constructing private algorithms or procedures by adding designed random noise, such that the output has a similar distribution regardless of whether the data are present for each individual participant, thereby helping to protect individuals' information in the dataset.

In this article, we study the problem of differentially private sequential change detection when both pre-change and post-change distributions are known and specified. 
We first extend the classical notion of $\epsilon$-differential privacy \cite{dwork2006calibrating} from fixed databases to sequential change detection tasks over potentially infinite data streams, where the sample size is not fixed.
We then develop a differentially private (DP) variant of the CUSUM procedure, called \hbox{DP-CUSUM}, which satisfies the new $\epsilon$-DP constraint. Our proposed DP-CUSUM procedure involves computing the CUSUM statistics while adding an independent Laplace noise at each time step, as well as to the detection threshold. By doing this, we do not release the CUSUM statistics or the threshold directly during the detection procedure. Meanwhile, the DP-CUSUM retains the same computational efficiency as the classical CUSUM procedure, as its computational overhead is minimal when noise is added at each time step. 

Our main contributions can be summarized as follows. First, we introduce a new concept of $\epsilon$-DP for sequential detection procedures when the sample size is not fixed. Second, we construct the DP-CUSUM procedure that satisfies the new $\epsilon$-DP constraint. Third, under the assumption that the log-likelihood ratio statistic is always bounded by a known constant, we prove a nonasymptotic lower bound to the average run length (ARL) of DP-CUSUM, which enables an analytical way of selecting the detection threshold, and a nonasymptotic upper bound to Lorden's worst-case detection delay (WADD) \cite{Lorden1971}. These nonasymptotic results also imply the asymptotic optimality of our proposed DP-CUSUM procedure under weak privacy constraints. Fourth, we extend our results to scenarios with general unbounded log-likelihood ratios by relaxing to a slightly weaker differential-privacy definition that adds a small constant $\delta$ to the DP definition. Finally, we validate our theoretical findings through numerical simulations, illustrating how the choice of privacy parameter affects the delay-versus-false-alarm trade-off in practical settings.

The remainder of this paper is organized as follows. Section\,\ref{sec:setup} provides preliminaries about sequential change detection when both pre-change and post-change distributions are fully specified, and introduces the modified concept of $\epsilon$-DP for sequential detection procedures. Section\,\ref{sec:dpcusum} presents our proposed DP-CUSUM detection procedure and the theoretical analysis when the log-likelihood function is bounded. Section\,\ref{sec:general} extends the method and theory to the general case with unbounded log-likelihood ratios. Finally, Section\,\ref{sec:numerical} presents simulation examples that demonstrate the performance of the proposed DP-CUSUM procedure under distributions with both bounded and unbounded log-likelihood ratios. All proofs are presented in the Appendix.

\subsection{Related Work}

The study of sequential change detection can be traced back to the early work of Page \cite{page-biometrica-1954} and has been studied extensively for several decades. Most works address the detection problem under the assumption of independent observations, particularly when new methods and theories are first introduced, but significant progress has also been made in extending these methods and theories to more complicated data models; see \cite{detectAbruptChange93,poor-hadj-QCD-book-2008,Siegmund1985,tartakovsky2014sequential,tartakovsky2019sequential,veeravalli2013quickest, tutorial_jsait} for thorough reviews in this field.

The first optimality result for sequential change detection appears in \cite{Shiryaev1963}, which focuses on detecting a change in the drift of a Brownian motion. It was studied under a Bayesian framework, where the changepoint was modeled as an independent, exponentially distributed random variable. In contrast, the non‑Bayesian (minimax) setting considers the changepoint to be deterministic but unknown. Under the minimax setting, the CUSUM procedure is arguably the most widely adopted change detection algorithm for the classical setup of detecting a change from a known distribution to a known alternative \cite{page-biometrica-1954}. The CUSUM procedure was first proved to be asymptotically optimum in \cite{Lorden1971} when observations are i.i.d.~before and after the change. Its exact optimality under the same data model was established in \cite{mous-astat-1986}. In addition to its strong optimality guarantees, the CUSUM statistic admits a recursive update, making it computationally efficient for sequential settings that require processing each new sample immediately. Although the problem of change detection when both pre- and post-change distributions are known has been extensively studied in the literature, these procedures do not consider data privacy-preserving guarantees.

First introduced by Dwork et al. \cite{dwork2006calibrating}, differential privacy has
garnered much attention. Many classical statistical procedures
are tailored to satisfy the privacy guarantees \cite{cai2021cost,avella2020privacy,degue2018differentially,canonne2019structure,zhang2024detection}. In particular, in the area of sequential change detection, an $\epsilon$-DP procedure was proposed in \cite{cummings2018differentially} to estimate the change point for sequential data when both pre-change and post-change are fully specified. This method relies on a fixed-size sliding window and applies an offline change-point estimator within each window, which may result in higher memory and computational overhead. Extensions to unknown post-change distributions have been considered in the univariate setting in \cite{cummings2020privately}. More recent efforts have explored broader data modalities and definitions of privacy. For instance, in \cite{berrett2021locally}, sequential change detection for multivariate nonparametric regression was studied under local differential privacy. In \cite{li2022network}, the authors extended the analysis of privatized networks from static to dynamic, underscoring the complexities and challenges of preserving privacy in the dynamic analysis of network data. However, none of these work provide the theoretical detection performance of their proposed private procedures in terms of average run length and the worst-case detection delay, which are two fundamental but important properties for the sequential change detection problem.

\section{Preliminaries and Problem Setup}\label{sec:setup}

In this section, we provide the necessary preliminaries and background on the sequential change detection problem and differentially private tools, then introduce the problem setup for differentially private sequential change detection.

\subsection{Basics of Classical Sequential Change Detection}\label{sec:basic_detection}

Suppose we observe data stream $\{X_t, t \in \mathbb N\}$, where each $X_t \in \Real^k.$  Initially, observations are independently and identically distributed (i.i.d.) following the probability density function (pdf) $f_0.$ At some unknown time $\tau$, an event occurs and changes the distribution of the data after the time to a distinct pdf $f_1.$ That is,
\begin{equation}
X_t  \stackrel{\rm i.i.d.}{\sim}\left\{\begin{array}{ll}
f_0,&t = 1,2,\ldots,\tau,\\[2pt]
f_1,&t = \tau+1,\tau+2,\ldots
\end{array}\right. 
\label{eq:data_model}
\end{equation}
In this article, we assume that $f_0$ and $f_1$ are both {\it known}.
Here, $\tau\in\{0,1,2,\ldots\}$ is a deterministic but unknown change time, and the goal is to raise the
alarm as quickly as possible after the change has occurred, while properly controlling the false alarm rate \cite{tartakovsky2014sequential,poor-hadj-QCD-book-2008,veeravalli2013quickest}.

A sequential change detection procedure consists of a {\it stopping time} $T$, denoting the time at which we stop and declare that a change has occurred before time $T$. Here $T$ is an integer-valued random variable, and the decision $\{T=t\}$ is based only on the observations in the first $t$ time steps. That is, $\{T=t\} \subseteq \cF_t,$ where we define the filtration $\cF_t=\sigma\{X_1,\ldots,X_t\}$ and let $\cF_0$ denote the trivial sigma-algebra. 
To evaluate the performance of $T,$ we further denote by $\Pro_\infty$ and $\Exp_\infty$ the probability measure and corresponding expectation when all samples follow the pre-change distribution $f_0$ (i.e.,~the change occurs at $\infty$). Similarly, we use $\Pro_0$ and $\Exp_0$ to represent the probability measure and expectation when all samples follow the post-change distribution $f_1$ (i.e.,~the change occurs at 0). More generally, we denote by $\Pro_\tau$ and $\Exp_\tau$ the probability measure and expectation when the change happens at time $\tau$. Under the classical minimax formulation for the sequential change detection problem \cite{Lorden1971,mous-astat-1986}, the optimal detection procedure is the one solving the following constrained optimization problem:
\begin{equation}
\begin{aligned} &\inf_{T} \,\text{WADD}(T):=\sup_{\tau\geq0}\,\essup\Exp_{\tau}[(T-\tau)^{+}|\cF_\tau]\\&\text{subject~to:}~\Exp_\infty[T]\geq\gamma>1.
\end{aligned}
\label{eq:optCUSUM}
\end{equation}

That is, among all stopping times that have an average false alarm period (also known as {\it average run length}, ARL) no smaller than a pre-specified constant $\gamma>1$, the optimal procedure should have the smallest \textit{worst-case average detection delay} (WADD). Here we adopt Lorden's definition of WADD, which takes the supremum, over all possible changepoints $\tau$, of the expected detection delay conditioned on the worst possible realizations before the change.

In the literature, it has been shown that the classical CUSUM procedure can solve the optimization problem (\ref{eq:optCUSUM}) \cite{Lorden1971,mous-astat-1986}. To be more concrete, the CUSUM statistic $\{S_t, t\ge 1\}$ corresponds to the maximum log-likelihood ratio over all possible changepoints up to time $t$ and can be calculated by a recursive form:
\begin{equation}\label{eq:1}
S_t = \max_{1\leq k\le t }\sum_{j=k}^t \ell(X_j)=\max\left(0,S_{t-1}\right) + \ell(X_t),
~S_0=0,
\end{equation}
where $\ell(X)=\log(f_1(X)/f_0(X))$ is the log-likelihood ratio (LLR) function  between $f_1$ and $f_0$. The CUSUM procedure is then defined as the first time when the CUSUM statistic exceeds some pre-defined
threshold $b$. That is, the CUSUM procedure is given by
\begin{equation}
T(b)=\inf\{t>0: S_t\geq b\}.
\label{eq:stCUSUM}
\end{equation}
 
For completeness, we provide in the following Lemma an asymptotic expression for the performance of the CUSUM procedure, which also serves as the information-theoretic lower bound for the WADD in problem \eqref{eq:optCUSUM}. The proof of the following Lemma can be found in \cite[Lemma 1]{xu2021optimum}.

\begin{lemma}[Performance of exact CUSUM \cite{xu2021optimum}]\label{lem:1}
For threshold $b=b_{\gamma}=\log\gamma$, the CUSUM procedure in \eqref{eq:stCUSUM} satisfies 
\begin{equation}
\Exp_\infty[T(b_{\gamma})]\geq \gamma,~~~~\text{WADD}[T(b_{\gamma})]=\frac{\log\gamma}{I_0} (1+o(1)),
\label{eq:perfCUSUM}
\end{equation}
where $I_0=\Exp_0\left[\ell(X)\right]$ is the Kullback-Leibler information number (divergence) of the post- and pre-change distributions.
\end{lemma}
From Lemma \ref{lem:1} we conclude that by applying the CUSUM procedure defined in \eqref{eq:stCUSUM} with threshold $b=\log\gamma$, the corresponding CUSUM stopping time $T$ enjoys an asymptotic performance captured by \eqref{eq:perfCUSUM}. By the optimality of the CUSUM procedure \cite{mous-astat-1986}, no other stopping time $T'$ that satisfies the same false alarm constraint can have a limiting value for the ratio $\mathbb E_0[T']/\frac{\log\gamma}{I_0}$ that is smaller than 1 as $\gamma\to\infty$.

\subsection{Differentially Private Sequential Change Detection}\label{sec:dp_detection}

While the classical CUSUM procedure is optimal for detecting distributional changes under the assumption of full data access, it may be unsuitable for applications in which individual-level data must remain private. In such scenarios, it is desirable to design change detection algorithms that also preserve the privacy of individuals contributing to the data stream. A principled way to achieve this is through the framework of differential privacy (DP) \cite{dwork2006calibrating,dwork2014algorithmic}.

Let us first review the classical definition of $\epsilon$-DP algorithms in the literature. 
Let $\mathcal{X}$ be the data domain (e.g., $\mathcal{X}=\mathbb{R}^k$), and let a database $D=\{x_1,\ldots,x_n\}\in\mathcal{X}^n$ consist of $n$ entries drawn from $\mathcal{X}$. Considering a random algorithm that maps from the database space $\mathcal{X}^n$ to $\mathbb{R},$ we say the algorithm is differentially
private if the outputs have similar distributions for neighboring datasets that we want to make it hard to distinguish. Here, two databases $D,D'$ are neighboring if they differ in at most {\it one} entry. We now introduce the formal definition of differential privacy \cite{dwork2006calibrating}.

\begin{definition}[$\epsilon$-DP]
A randomized algorithm $\mathcal{A}:  \mathcal{X}^n \to \mathbb{R}$ is $\epsilon $-differentially private (DP) if for every pair of neighboring databases $D, D'\in \mathcal{X}^n,$  and for every subset of possible events $\mathcal S\subseteq \mathbb{R},$ $\Pro_{\mathcal{A}}(\mathcal{A}(D)\in \mathcal S)\le e^{\epsilon}\Pro_{\mathcal{A}}(\mathcal{A}(D')\in\mathcal S).$ 
\end{definition}

We should emphasize that in the definition of $\epsilon$-DP, the expectation is taken over the randomness of the algorithm $\mathcal A,$ while the two databases $D$ and $D'$ are fixed. One common technique to achieve $\epsilon$-differential privacy is
by adding a Laplace noise \cite{dwork2014algorithmic}. Specifically, for a real-valued deterministic function $L:\mathbb{R}^n \to \mathbb{R},$ define its sensitivity as $\Delta (L)=\max_{D,D' \text{are neighbors}}|L(D)-L(D')|.$ Then, a randomized $\epsilon$-DP algorithm can be obtained by adding an independent Laplace random noise $\text{Lap}(\Delta(L)/\epsilon)$ to the realization of the
statistic $L(D).$ That is, $\tilde{L}(D)=L(D)+\text{Lap}(\Delta(L)/\epsilon)$ is a randomized algorithm satisfying the $\epsilon$-DP constraint. Here, we denote $\text{Lap}(\beta)$ as the Laplace random variable with zero mean and scale parameter $\beta>0$, with probability density function $f_{\text{Lap}(\beta)}(x)=\exp(-|x|/\beta)/(2\beta).$ 

In particular, for our problem of sequential change detection with
positive integer output in $\mathbb Z^+,$ we can modify the general definition of $\epsilon$-DP to obtain the following definition of $\epsilon$-DP sequential detection procedure. Let $X_{(1:t)} = (X_1, \ldots, X_t)$ denote the sequence of observations up to time $t$, and let $X'_{(1:t)}$ be a neighboring dataset that differs from $X_{(1:t)}$ in exactly one entry (i.e., they differ in at most one $X_i$ for some $1\le i\le t$). We now present the formal definition below.

\begin{definition}[$\epsilon$-DP Sequential Detection Procedure]\label{dpdef}
A randomized sequential change detection procedure with stopping time $\mathcal T$ is said to be $\epsilon$-differentially private, if for every pair of neighboring data streams $X_{(1:n)}, X'_{(1:n)}$ (differing in at most one entry), the distribution over the randomized stopping time $\mathcal T$ satisfies the differential privacy constraint, 
\begin{equation}\label{eq:dp_def} 
\Pro_{\mathcal T}(\mathcal T =n \mid X_{(1:n)}) \leq e^\epsilon \Pro_{\mathcal T}(\mathcal T =n \mid X'_{(1:n)}), \!~ \forall \, n\ge 1, 
\end{equation}
where the probability $\Pro_{\mathcal T}$ is taken over the randomness in $\mathcal T.$
\end{definition}

Here, a larger value of $\epsilon$ implies a weaker privacy constraint. Intuitively, the above definition means that altering any single observation in the data stream only slightly affects the distribution of the randomized stopping time $\mathcal T$. Therefore, one cannot easily infer individual data values from the detection output $\mathcal T$. Our goal is to devise a sequential detection procedure that satisfies the $\epsilon$-DP constraint in \eqref{eq:dp_def} while maintaining good detection performance with a small detection delay when the average run length is controlled. In the next section, we propose to first construct an $\epsilon$-DP sequential detection procedure based on the classical CUSUM procedure. We then study the theoretical detection properties of our proposed procedure, with emphasis on how the privacy-constraint parameter $\epsilon$ affects its detection performance.

\section{Proposed DP-CUSUM Procedure with Bounded Log-Likelihood Ratio}\label{sec:dpcusum}

In this section, we design a differentially private variant of the CUSUM procedure that satisfies the sequential differential privacy constraint in (\ref{eq:dp_def}) while retaining strong detection performance guarantees. We first define the sensitivity of the log-likelihood function $\ell(X)$ as follows.
\begin{definition}[Sensitivity of $\ell$]\label{def:sensi}
The sensitivity of the log-likelihood ratio function $\ell(\cdot)$ is defined as
$$\Delta=\Delta(\ell):=\sup_{x,y\in\mathbb R^k} \left|\ell(x)-\ell(y)\right|.$$ 
\end{definition}

In this section, we first consider the simpler case where $\Delta$ is bounded. This assumption holds for a broad class of distributions, including any pair of pre‑ and post‑change distributions that are both discrete with the same support (e.g., Bernoulli, Binomial) or certain continuous distributions (e.g., Laplace). We will extend our method and analysis to more general distributions where $\Delta$ is unbounded, under a relaxed differential privacy constraint in Section \ref{sec:general}.

We define the randomized CUSUM statistic $\{\tilde{S}_t\}_{t \ge 1}$ recursively as 
\begin{equation}
\tilde{S}_t=S_t+Z_t, \label{eq:dp-cusum-stat} 
\end{equation}
where $S_t$ is the classical CUSUM statistics from \eqref{eq:1}, and {$Z_t\sim \text{Lap}(\frac{2\Delta}{\epsilon})$} are i.i.d Laplace noise at each time step $t$ and they are also independent to all of data sequence $\{X_t, t \in \mathbb N\}$. 
Then, our proposed private detection procedure is defined as 
\begin{align}
\tilde{T}(b)&=\inf\{t>0: \tilde{S}_t\geq b+W\}, \label{eq:dp-cusum-stop}
\end{align}
where $W\sim \text{Lap}(\frac{2\Delta}{\epsilon})$ is an independent Laplace noise term added to the threshold, fixed upon generation, and $b$ is a pre-specified deterministic threshold chosen by satisfying the false alarm rate constraint. The DP-CUSUM procedure is summarized in Algorithm \ref{alg:dp-cusum}. Note that the DP‑CUSUM procedure incurs only one additional Laplace‑noise sample per time step beyond the classical CUSUM updates. Hence, its overall time complexity remains $O(t)$ as in the non-private case.

\begin{algorithm}[t!]
\caption{DP-CUSUM Procedure}
\label{alg:dp-cusum}
\begin{algorithmic}[1]
\REQUIRE Data sequence $\{X_t, t\in\mathbb{N}\}$, privacy parameter $\epsilon$, sensitivity $\Delta$ of LLR $\ell$, threshold $b$.
\ENSURE Stopping time $\tilde{T}(b)$.
\STATE Sample $W \sim \text{Lap}\left(\frac{2\Delta}{\epsilon}\right)$.
\STATE Initialize $t \leftarrow 0$, $S_0 \leftarrow 0$, $\tilde S_0 \leftarrow -\infty$.
\WHILE{$\tilde S_t < b+W$}
    \STATE $t \leftarrow t+1$ and observe a new data $X_{t}$.
    \STATE Update classical CUSUM statistic $S_t = \max(0, S_{t-1}) + \ell(X_t)$.
    \STATE Sample $Z_t \sim \text{Lap}\left(\frac{2\Delta}{\epsilon}\right)$.
    \STATE Compute privatized statistic $\tilde{S}_t = S_t + Z_t$.
\ENDWHILE
\STATE Output stopping time $\tilde{T}(b) = t$; declare that a change has occurred before time $\tilde{T}(b)$.
\end{algorithmic}
\end{algorithm}

We then present the following theorem, which guarantees that the proposed DP-CUSUM procedure, characterized by its stopping time $\tilde{T}(b)$, satisfies the $\epsilon$-differential privacy condition in (\ref{eq:dp_def}). The proof follows the idea of AboveThreshold in \cite{dwork2014algorithmic} and can be found in Appendix~\ref{app:proofs}.

\begin{theorem}[DP Guarantee]\label{thm:dp}
Assume $\Delta$ is bounded. The DP-CUSUM procedure $\tilde{T}(b)$ is $\epsilon$-DP satisfying (\ref{eq:dp_def}).   
\end{theorem}

It is worthwhile mentioning that, compared to the existing private change detection procedure proposed in \cite{cummings2018differentially}, we require a smaller Laplace noise $Z_t,W\sim\text{Lap}(\frac{2\Delta}{\epsilon})$ added to both the CUSUM detection statistics and detection threshold, whereas the method in \cite{cummings2018differentially} added $Z_t\sim \text{Lap}(\frac{8\Delta}{\epsilon})$ to the detection statistics and $W\sim \text{Lap}(\frac{4\Delta}{\epsilon})$ to detection threshold. Such reduced noise can significantly improve the detection performance while still achieving $\epsilon$-DP, since smaller Laplace noise means less degradation in detection performance, as validated in our numerical comparisons in Section~\ref{sec:numerical}. Intuitively, our proposed method achieves $\epsilon$-DP under a smaller noise scale because the DP-CUSUM procedure is analyzed as a single, unified mechanism without invoking privacy composition. We allocate the full privacy budget $\epsilon$ to the entire detection process and exploit the structural property of the recursive CUSUM statistic. In contrast, the method in \cite{cummings2018differentially} adopts a two-stage design that combines the {\it AboveThreshold} mechanism with the offline {\it Report Noisy Max} procedure. Its DP guarantee relies on the direct composition theorem, where each component consumes half of the total privacy budget ($\epsilon/2$). Consequently, to guarantee overall $\epsilon$-DP, the {\it AboveThreshold} component alone must use a larger noise scale corresponding to its allocated privacy budget of $\epsilon/2$.

Then, we analyze the theoretical detection performance of the proposed DP-CUSUM procedure $\tilde T(b)$. Note the randomized detection procedure $\tilde{T}(b)$ also depends on the added Laplace random variables $Z_t$ and $W.$ That is, the decision $\{\tilde{T}=t\}$ is based on the observations $\{X_1, \cdots, X_t\}$ and the corresponding added Laplace noise $\{Z_1,\cdots, Z_t, W\}.$ Thus, throughout this paper, when analyzing the randomized detection procedure $\tilde{T}(b)$, we denote by $\Pro_\infty$ and $\Exp_\infty$ the joint probability measure and corresponding expectation when all samples $\{X_t\}_{t\ge 1}$ are i.i.d and follow the pre-change distribution $f_0$ (i.e.,~the change occurs at $\infty$), $\{Z_t\}_{t\ge 1}$ are i.i.d with Lap($\frac{2\Delta}{\epsilon}$), and $W$ follow Lap($\frac{2\Delta}{\epsilon}$). Similarly, we use $\Pro_0$ and $\Exp_0$ to represent the joint probability measure and expectation when all samples $\{X_t\}_{t\ge 1}$ follow the post-change distribution $f_1$ (i.e.,~the change occurs at 0) and $\{Z_t\}$ and W remain the same Laplace distribution. More generally, we denote by $\Pro_\tau$ and $\Exp_\tau$ the joint probability measure and expectation when the change happens at time $\tau$. Then, the ARL and WADD of our proposed procedure is defined by $\Exp_{\infty}[\tilde T(b)]$ and $\sup_{\tau\geq0}\,\essup\Exp_{\tau}[(\tilde{T}(b)-\tau)^{+}|X_1,\cdots, X_{\tau}]$ respectively.

Then, the following theorem provides a lower bound to the ARL of the DP-CUSUM procedure. 

\begin{theorem}[ARL of DP-CUSUM] \label{thm:arl}
Assume $\Delta$ is bounded. For any $\epsilon > 0$ and threshold $b>2,$ define the following quantity \begin{equation}\label{eq:h}
h(\epsilon,\Delta) = \min\{\frac{\epsilon}{2\Delta}, 1\}.
\end{equation} 
Then we have 
\begin{equation}\label{eq:arl}
\Exp_{\infty}[\tilde{T}(b)]\ge \frac{e^{h(\epsilon,\Delta) b -2}}{4(b+1)^2}.
\end{equation}
\end{theorem}
The lower bound in \eqref{eq:arl} provides an analytical method for selecting the threshold without extensive simulation. That is, for any desired ARL $\gamma$, one can numerically solve for $b$ such that the lower bound in \eqref{eq:arl} is no smaller than $\gamma$.
Therefore, by Theorem \ref{thm:arl}, we can see by setting the threshold $b$ such that
\begin{equation}\label{eq:threshold}
\frac{e^{h(\epsilon,\Delta) b -2}}{4(b+1)^2} = \gamma \Rightarrow b =b_{\gamma}= \frac{\log\gamma}{h(\epsilon,\Delta)}(1+o(1)),
\end{equation}
our proposed procedure $\tilde T(b_{\gamma})$ satisfies the ARL constraint $\Exp_{\infty}(\tilde{T}(b_{\gamma}))\ge \gamma.$

We next analyze the worst-case average detection delay (WADD) of DP-CUSUM. In the following Lemma, we first show that the worst case is attained when the change occurs at time $\tau=0$, which is a useful property also enjoyed by the classical CUSUM procedure. Due to such property, it suffices to compute $\Exp_0[\Tilde T(b)]$ when evaluating the worst-case delay, without the need to enumerate all possible changepoint locations.

\begin{lemma}[Worst-Case Average Detection Delay]\label{lem:wadd0}
For any  $b\geq0$, we have that
\begin{equation}
\text{WADD}(\tilde T(b)) \leq \Exp_0[\tilde T(b)].    
\end{equation}
\end{lemma}

We are now ready to analyze an upper bound for the WADD of our proposed DP-CUSUM procedure. By Lemma \ref{lem:wadd0}, it suffices to establish an upper bound on $\Exp_0[\tilde{T}(b)]$. However, deriving such a bound is nontrivial due to the update rule, which includes both truncation at zero and the addition of Laplace noise at each time step. Since we only aim to find a suitable upper bound to the delay, we introduce an alternative stopping time in the proof that is guaranteed to incur a longer delay than $\tilde{T}(b)$ but is easier to analyze. The details are provided in Appendix~\ref{app:proofs}.

\begin{theorem}[WADD of DP-CUSUM] \label{thm:edd}
Assume $\Delta$ is bounded. We have for any $b>0$, 
\begin{equation}\label{eq:wadd}
\Exp_0[\Tilde T(b)]  \le \frac{b}{I_0} + \frac{4\Delta}{I_0^{3/2}\epsilon}\sqrt{b}+ C,
\end{equation}  
where $I_0=\Exp_0\left[\ell(X)\right]$ is the Kullback-Leibler information number (divergence) of the post- and pre-change distributions, and $C$ is a constant that does not depend on the threshold $b$ but depends on $\Delta,\epsilon, I_0.$
\end{theorem}

Setting the threshold $b=b_{\gamma}$ as in (\ref{eq:threshold}), our proposed DP-CUSUM $\tilde T(b_{\gamma})$ satisfies  $\Exp_{\infty}(\tilde{T}(b_{\gamma}))\ge \gamma$ and, by Lemma~\ref{lem:wadd0} and Theorem~\ref{thm:edd}, its WADD satisfies 
\begin{align}\label{eq:bd}
\text{WADD}(\Tilde T(b_\gamma)) = \Exp_0[\Tilde T(b_\gamma)] \leq \frac{\log\gamma}{h(\epsilon,\Delta) I_0}(1+o(1)).
\end{align}
Note $h(\epsilon,\Delta) = \min\{\frac{\epsilon}{2\Delta}, 1\}$. When $\epsilon \ge 2\Delta$, our proposed DP-CUSUM yields the same order of detection delay as the CUSUM procedure in Lemma \ref{lem:1}. This result implies that, under the relatively weaker privacy requirement ($\epsilon \ge 2\Delta$), our proposed DP-CUSUM procedure achieves first-order {\it asymptotic optimality} while maintaining the privacy constraint, i.e., its detection delay asymptotically matches that of the classical CUSUM as ARL grows.

However, when $\epsilon<2\Delta$, the upper bound we established for the WADD of DP-CUSUM becomes $\frac{\log\gamma}{ I_0}(\frac{2\Delta}{\epsilon}) (1+o(1))$, which is greater than the optimal detection delay $\frac{\log\gamma}{I_0} (1+o(1))$ of the classical (non-private) CUSUM. This implies that our proposed DP-CUSUM will still preserve data privacy under the stronger privacy constraint, albeit at the cost of some detection efficiency. This illustrates the fundamental trade-off between increased detection delay and privacy gain. Specifically, we can use the privacy parameter $\epsilon$ to quantify the privacy gain of a differentially private method. A lower \(\epsilon \) value indicates a higher level of privacy, while a higher \(\epsilon \) value provides less privacy but greater accuracy. The exact CUSUM procedure given in Eq. \eqref{eq:stCUSUM} can be viewed as the extreme case with $\epsilon=+\infty$ (i.e., no privacy at all), as the exact CUSUM procedure does not satisfy any finite differential privacy guarantee. The WADD in Eq.~\eqref{eq:bd} establishes an explicit privacy–delay trade-off that links the delay directly to $\epsilon$, where the delay approaches that of the exact CUSUM as $\epsilon$ becomes sufficiently large. 
Moreover, it remains an open question what the tight information-theoretic lower bound for the WADD is under the $\epsilon$-DP constraint \eqref{eq:dp_def} when $\epsilon < 2\Delta$, and whether DP-CUSUM retains asymptotic optimality in this regime. We leave this as a direction for future work.

\begin{remark}[Relaxed DP for $n>n_0$] The $\epsilon$-DP requirement in Eq.~\eqref{eq:dp_def} may appear restrictive when the sample size $n$ is small (e.g., $n=1$). In some cases, privacy protection is required only for data collected after a certain point in time—for example, when a sequence begins with publicly available or non-sensitive warm-start data. In such settings, the DP condition in Eq.~\eqref{eq:dp_def} can be relaxed to hold only for all $n > n_0$ for some positive integer $n_0$, thereby enforcing privacy guarantees only for newly arriving data. In such cases, our DP-CUSUM procedure in Algorithm~\ref{alg:dp-cusum} could be modified accordingly so that Laplace noise is added only when $n > n_0.$ That is, $\tilde S_t  = S_t + Z_t \ind{t>n_0}$. Note that in such cases, the proof techniques of Theorems~\ref{thm:arl} and~\ref{thm:edd} remain valid with only minor modifications that do not affect the asymptotic order of the ARL and WADD as $b \to \infty$, provided that $n_0$ is a fixed integer independent of $b$. Therefore, our asymptotic results and the argument for asymptotic optimality remain unchanged under this relaxed DP requirement.
\end{remark}

\section{Proposed \algname{} Procedure with Unbounded Log-Likelihood Ratio}\label{sec:general}

In the previous sections, we assumed that the LLR $\ell(x)$ is bounded, which ensures a finite global sensitivity $\Delta<\infty$ and allows the DP-CUSUM procedure proposed in \eqref{eq:dp-cusum-stat} and \eqref{eq:dp-cusum-stop} to achieve $\epsilon$-differential privacy. However, in many practical scenarios—particularly when dealing with continuous distributions such as Gaussian or exponential families—the LLR $\ell(x)$ may be unbounded, and the global sensitivity $\Delta$ becomes infinite. 

To address this, in this section, we adopt a relaxed version of differential privacy by allowing the user to select a small probability $\delta$ of failure, resulting in a relaxed $(\epsilon, \delta)$-differential privacy definition \cite{cummings2018differentially}. In this setting, we continue to use the DP-CUSUM procedure but modify the sensitivity parameter in Algorithm~\ref{alg:dp-cusum} to depend on the underlying data distributions and the user-defined privacy constraint $\delta$. This distribution-dependent choice ensures that the relaxed $(\epsilon, \delta)$-differential privacy condition is satisfied, as formalized below. To distinguish it from the vanilla DP‑CUSUM in the finite-$\Delta$ case, we refer to this modified procedure as \algname{}.

We first present the \algname{} procedure when the sensitivity of $\ell(x)$ is unbounded. For a given $\delta>0$, define 
\begin{align}\label{eq:Adelta}
A_{\delta}=\inf\{t:\max_{i=0,1}\,\Pro_{X\sim f_i}\left( 2 |\ell(X)|\ge t\right)\le \delta/2\}.
\end{align}
That is, $A_\delta$ guarantees that the event $2|\ell(X)| \ge A_\delta$ occurs with probability at most $\delta/2$ under both the pre- and post-change distributions. 
We then apply the DP-CUSUM procedure as described in the previous section or Algorithm~\ref{alg:dp-cusum}, with the sensitivity parameter $\Delta$ replaced by the distribution-dependent quantity $A_\delta$. The following theorem establishes that this modified procedure satisfies a relaxed $(\epsilon, \delta)$-differential privacy guarantee.

\begin{theorem}[($\epsilon,\delta$)-DP guarantee]\label{thm4}
When the LLR is unbounded, the stopping time $\tilde{T}$ in \eqref{eq:dp-cusum-stat} and \eqref{eq:dp-cusum-stop} with the choice of $\Delta=A_{\delta}$ satisfies for every $t \in \mathbb{N}$, and every pair of neighboring data streams $X_{(1:t)}, X'_{(1:t)}$ that differs in only one entry $X_k\neq X_k'$, the distribution over the stopping time $\tilde{T}$ satisfies the following weaker differential privacy constraint: 
\begin{multline} \label{eq:dp_def_delta}
\Pro_{\tilde T, X_k,X'_k}(\tilde{T} =t \mid X_{(1:t)\setminus k})  \\ \leq e^\epsilon \Pro_{\tilde T, X_k,X'_k}(\tilde{T} =t \mid X'_{(1:t)\setminus k}) + \delta,  
\end{multline}
where the probability is taken over the randomness in $\tilde{T}$ and the differing entry $X_k,X_k',$ while the rest of data entries are fixed.
\end{theorem}

We would like to emphasize that this relaxed $(\epsilon, \delta)$-differential privacy guarantee in \eqref{eq:dp_def_delta} is not the same as the classical definition of $(\epsilon, \delta)$-DP in literature \cite{dwork2006calibrating}. In the classical definition, the entire database---including the differing entry---is fixed, and the probability is taken only over the internal randomness of the algorithm, that is, $\Pro_{\mathcal{A}}(\mathcal{A}(D)\in\mathcal{S})\le e^{\epsilon}\Pro_{\mathcal{A}}(\mathcal{A}(D')\in\mathcal{S})+\delta$ for all neighboring datasets $D,D'$ and measurable sets $\mathcal{S}$. In contrast, our slightly modified formulation \eqref{eq:dp_def_delta} treats the differing entry itself as random and introduces a small probability $\delta$ to capture rare failure events where the LLR at differing entries exceeds a bounded threshold, i.e., $|\ell(X_k)-\ell(X_k')|>A_\delta$. Conditional on the high-probability event $|\ell(X_k)|\le A_\delta/2,|\ell(X_k')|\le A_\delta/2$, our procedure still achieves $\epsilon$-DP, and $\delta$ thus accounts for the probability of violating this condition. The same adaptation has also been used in \cite{cummings2018differentially} for handling unbounded LLRs in private change detection. To avoid confusion, we named our proposed procedure as \algname{}  procedure instead of $(\epsilon,\delta)$-DP-CUSUM procedure.

We then extend our analysis in Section \ref{sec:dpcusum} to the current setting where the LLR $\ell(\cdot)$ is unbounded. Following similar proof strategies as in Section \ref{sec:dpcusum}, we derive a lower bound for the ARL and an upper bound for the WADD of our \algname{} procedure, which satisfies the relaxed $(\epsilon, \delta)$-differential privacy guarantee.

\begin{corollary}[ARL of \algname{}] \label{thm:arl2}
For the general case with unbounded LLR, for any $\epsilon > 0$, $\delta\in(0,1),$ $b>2,$ and the corresponding $A_\delta$ defined in \eqref{eq:Adelta}, 
we have 
\begin{equation}\label{eq:arl2}
\Exp_{\infty}[\tilde{T}(b)]\ge \frac{e^{h(\epsilon,A_\delta) b -2}}{4(b+1)^2},
\end{equation}
where $h(\epsilon, A_\delta)=\min\{\frac{\epsilon}{2A_\delta}, 1\}$.
\end{corollary}

\begin{proof}
Note that the proof of Theorem \ref{thm:arl} does not require the log-likelihood ratio $\ell(\cdot)$ to be bounded. Consequently, the same argument still applies by replacing $\Delta$ with $A_\delta$.    
\end{proof}

\begin{corollary}[WADD of \algname{}] \label{thm:edd2}
For the general case with unbounded LLR, assume the log-likelihood $\ell(X)$ is $\sigma^2$-sub-Gaussian under both the pre- and post-change distributions. That is, for $i=0,1,$ and $\lambda\in\mathbb R,$
$$\Exp_{X\sim f_{i}}[e^{\lambda (\ell(X) - \Exp[\ell(X)]) }] \le e^{ \frac{\lambda^2\sigma^2}{2}}.$$ 
Then, for any $\epsilon > 0$, $\delta\in(0,1), b> 0,$ and the corresponding $A_\delta$ defined in \eqref{eq:Adelta}, we have the worst-case average detection delay satisfies
\begin{equation}\label{eq:wadd2}
\text{WADD}(\Tilde T(b))\le \frac{b}{I_0} + \frac{4 A_\delta}{I_0^{3/2}\epsilon}\sqrt{b}+ C,
\end{equation}  
where $C$ is a constant that does not depend on the threshold $b$ but depends on $\epsilon, \delta,I_0, \sigma^2$.
\end{corollary}

Similar to Section~\ref{sec:dpcusum}, the lower bound in \eqref{eq:arl2} can be used to select the threshold $b$ as
\begin{equation}\label{eq:threshold2}
b =b_{\gamma}=  \frac{\log\gamma}{h(\epsilon,A_\delta)}(1+o(1)),
\end{equation}
to guarantee the ARL constraint $\Exp_\infty[\tilde T] \ge \gamma$ is met. Under this choice of threshold, the resulting WADD of the \algname{} procedure is upper-bounded by
\begin{equation}\label{eq:wadd-asym2}
\text{WADD}(\Tilde T(b_{\gamma})) \le \frac{\log\gamma}{h(\epsilon,A_\delta) I_0} (1+o(1)),  
\end{equation}
which provides implications on how the privacy parameters $\epsilon$ and $\delta$ affect the detection performance of our proposed \algname{} procedure. In particular, note $h(\epsilon,A_{\delta})\le 1.$

Our proposed procedure will achieve the first-order {\it asymptotical optimality} as the classical CUSUM procedure if $\epsilon \ge 2A_{\delta}.$ On the other hand, if $\epsilon<2A_\delta$, the upper bound becomes $\frac{\log\gamma}{I_0} (\frac{2A_\delta}{\epsilon})(1+o(1))$, which is greater than the optimal detection delay $\frac{\log\gamma}{I_0} (1+o(1))$ of classical CUSUM. Therefore, a smaller value of $h$ implies greater degradation in detection performance due to the privacy requirement. We further provide a Gaussian example below to illustrate how the privacy parameters $\epsilon$ and $\delta$ affect the detection efficiency of our proposed procedure by visualizing the value of $h(\epsilon, A_{\delta}).$ We emphasize that $\delta$ is user-specified and it yields the trade-off between privacy stringency and detection delay. In practice, one can choose the largest $\delta$ acceptable under one’s privacy requirements to make $h$ closer to 1, and thus reduce the detection delay.

\begin{example}[Gaussian distributions.]
We provide an example under Gaussian distributions to demonstrate the theoretical effect of privacy parameters $\epsilon,\delta$ on detection delay. Assume one-dimensional Gaussian mean shift from the pre-change distribution $N(0,1)$ to the post-change distribution $N(\mu,1)$ with $\mu>0$. Therefore $f_0(x)=\frac{1}{\sqrt{2\pi}}e^{-\frac12x^2}$ is the pdf of pre-change distribution, and $f_1(x)=\frac{1}{\sqrt{2\pi}}e^{-\frac12(x-\mu)^2}$ is the pdf of post-change distribution. The LLR equals $\ell(X) = \log f_1(X)/f_0(X)=\mu X - \frac12\mu^2$. Taking $A_\delta=2 |\mu| \cdot z_{\delta/4} + \mu^2$ ensures $\max_{i=0,1}\,\Pro_{X \sim f_i}\left(2|\ell(X)|\ge A_\delta\right)\le \delta/2$. 
The quantity $h(\epsilon,A_\delta) = \min\{\frac{\epsilon}{2A_\delta},1\}$ is the effective privacy scaling factor that directly influences the asymptotic performance of the \algname{} procedure. In particular, the best-case scenario corresponds to $h(\epsilon,A_\delta) = 1$, under which the \algname{} procedure is guaranteed to be first-order asymptotically optimal. In contrast, smaller values of $h$ indicate a greater degradation in detection efficiency. 
Fig.~\ref{fig:heatmap_h} visualizes $h(\epsilon, A_\delta)$ over a wide range of $\epsilon$ and $\delta$ values for different magnitude of the mean shift $\mu\in\{0.1,0.25,0.5\}$. The plots show that $h(\epsilon, A_\delta)$ generally decreases as $\mu$ increases, because $A_\delta$ increases with $\mu$, thereby leading to a stronger impact of privacy on detection performance. 
The red dashed line in Fig.~\ref{fig:heatmap_h} serves as the boundary for optimal detection given by Theorem~\ref{thm:edd2}: any privacy parameters satisfying $\epsilon\ge 2A_\delta$ (equivalently $h(\epsilon,A_\delta)=1$) guarantee that the \algname{} procedure achieves first‑order asymptotic optimality.

\begin{figure}[!t]
    \centering \includegraphics[width=0.95\linewidth]{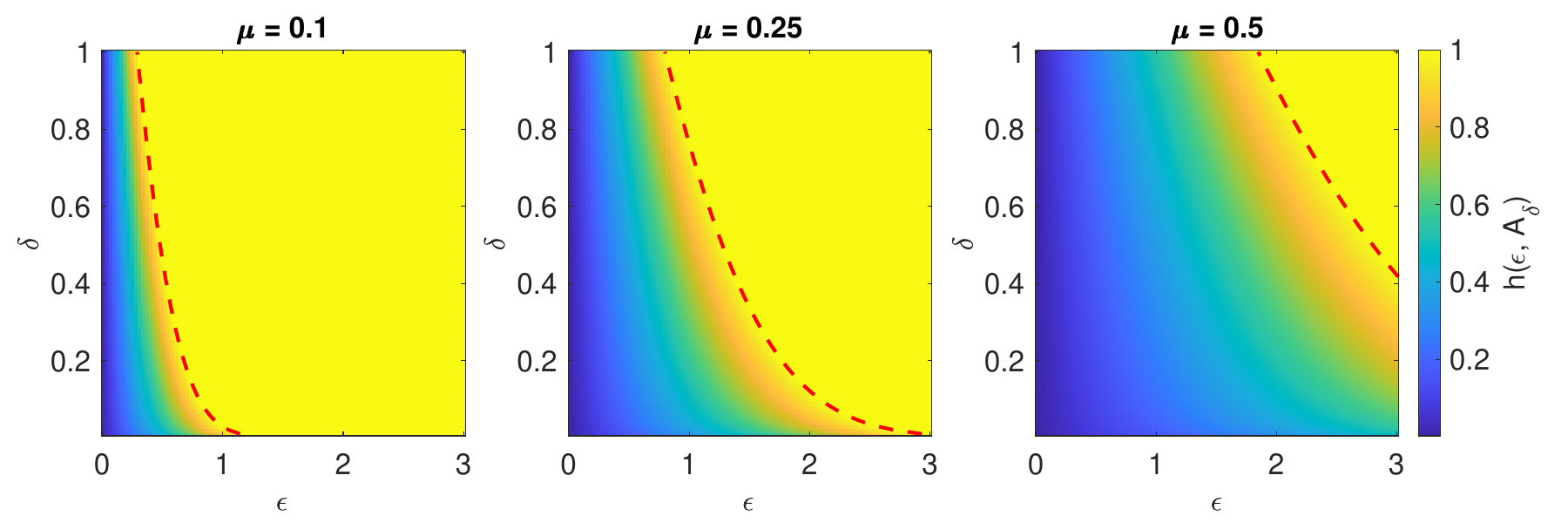}
    \vspace{-0.1in}
    \caption{Heatmaps of the effective privacy factor $h(\epsilon, A_\delta)$ under a Gaussian mean shift from $N(0,1)$ to $N(\mu,1)$, for $\mu$ in $\{0.1, 0.25, 0.5\}$. Each panel shows how $h$ varies with $\epsilon \in (0, 3)$ and $\delta \in (0, 1)$. The dashed red curve denotes $\epsilon = 2A_\delta$, i.e., $h(\epsilon, A_\delta) = 1$; above this curve, the DP-CUSUM procedure is proved to be asymptotically optimal by Eq.~\eqref{eq:wadd-asym2}.}
    \label{fig:heatmap_h}
\end{figure}
\end{example}

\begin{remark} From Fig.~\ref{fig:heatmap_h}, we observe that stronger signals require larger values of $\epsilon$ and $\delta$ to attain optimality. Under the Gaussian example, this is because $A_\delta=2 |\mu| \cdot z_{\delta/4} + \mu^2$ increases with $|\mu|$ and thus larger $\mu$ requires larger $\epsilon$ value to ensure optimality; conversely, under fixed $\mu$, smaller $\delta$ also leads to larger $A_\delta$ and thus requires larger $\epsilon$. It is worthwhile mentioning that this phenomenon is universal across general distributions. Intuitively, for stronger change signals (such as the case with larger $\mu$ in the Gaussian example), the LLR tends to take larger values with higher probability (and can become unbounded as assumed). Such a larger magnitude of LLR makes the CUSUM statistic more revealing of the raw data, necessitating stronger noise injection to preserve privacy---thereby degrading detection performance and requiring more relaxed privacy parameters $(\epsilon,\delta)$ for optimality. In contrast, for weaker change signals (such as smaller $\mu$ in the Gaussian setting), the LLR remains smaller and bounded with higher probability, requiring less noise and achieving asymptotic optimality even under smaller $(\epsilon,\delta)$. These observations also demonstrate the fundamental trade-off between detection delay and privacy gain: larger changes are easier to detect but require more effort to satisfy the DP guarantee. 
\end{remark}

\section{Numerical Results}\label{sec:numerical}

In this section, we evaluate the empirical performance of the proposed DP-CUSUM procedure under both bounded and unbounded log-likelihood ratio settings. For each configuration, we simulate the ARL under the pre-change regime and the WADD under the post-change regime.

\subsection{DP-CUSUM with Bounded LLR: Laplace Distribution}\label{num:lap}

We consider a mean shift in Laplace distribution family $\text{Lap}(\mu,1),$ which has a pdf $f(x|\mu)=\exp\left(-|x - \mu|\right)/2.$ Note we use a different notation here because we have used the shorthand notation $\text{Lap}(b)$ before for denoting the zero‑mean Laplace noise added in the detection procedure. 
We let the pre-change distribution be $\text{Lap}(0,1)$ and the post-change distribution be $\text{Lap}(\mu,1)$, with a nonzero post-change mean $\mu\neq 0$. Note the LLR between $\text{Lap}(\mu,1)$ and $\text{Lap}(0,1)$ is $\ell(x) = \left( |x| - |x - \mu| \right),$ which implies the sensitivity of the LLR is $\Delta = \sup_{x,x'\in \mathbb R} |\ell(x) - \ell(x')| = 2|\mu|$. 

We consider two different mean shift magnitudes $\mu=0.2$ and $\mu=0.5$. Under each mean value, we consider varying privacy parameters. Specifically, for $\mu=0.2$, we consider $\epsilon\in\{0.2,0.4,0.6,0.8,1\}$ and for $\mu=0.5$, we consider $\epsilon\in\{0.8,1,1.5,2\}$. These $\epsilon$ values are carefully chosen to ensure coverage around the critical threshold $\epsilon = 2\Delta$ to examine the transition in detection performance near the point of asymptotic optimality. 
To evaluate the detection performance, we implement the DP-CUSUM procedure and simulate both the average run length under the pre-change regime and the detection delay under the post-change regime. Each configuration is repeated 10,000 times to compute the average performance. We plot the results in Fig.~\ref{fig:lap-edd}.

In Fig.~\ref{fig:lap-edd}(a), it can be seen that the detection delays for $\epsilon = 0.8$ and $1$ are close to that of the exact CUSUM. This is consistent with our theoretical results on WADD in Theorem~\ref{thm:edd}. Specifically, for $\mu = 0.2$, we have $\Delta = 2\mu = 0.4$. By Eq.~(\ref{eq:bd}), the DP-CUSUM procedure is guaranteed to be asymptotically optimal when $\epsilon \ge 2\Delta=0.8$, which aligns with Fig.~\ref{fig:lap-edd}(a). As $\epsilon$ decreases, the delay increases, especially at larger ARL values, reflecting the trade-off between privacy and detection. A similar pattern is observed in Fig.~\ref{fig:lap-edd}(b) for $\mu = 0.5$, where $\Delta = 1$. It can be seen that the delay for $\epsilon = 2\Delta=2$ matches the exact CUSUM, while smaller $\epsilon$ leads to worse performance.

\begin{figure}[!t]
    \centering
    \begin{tabular}{cc}
      \includegraphics[width=0.48\linewidth]{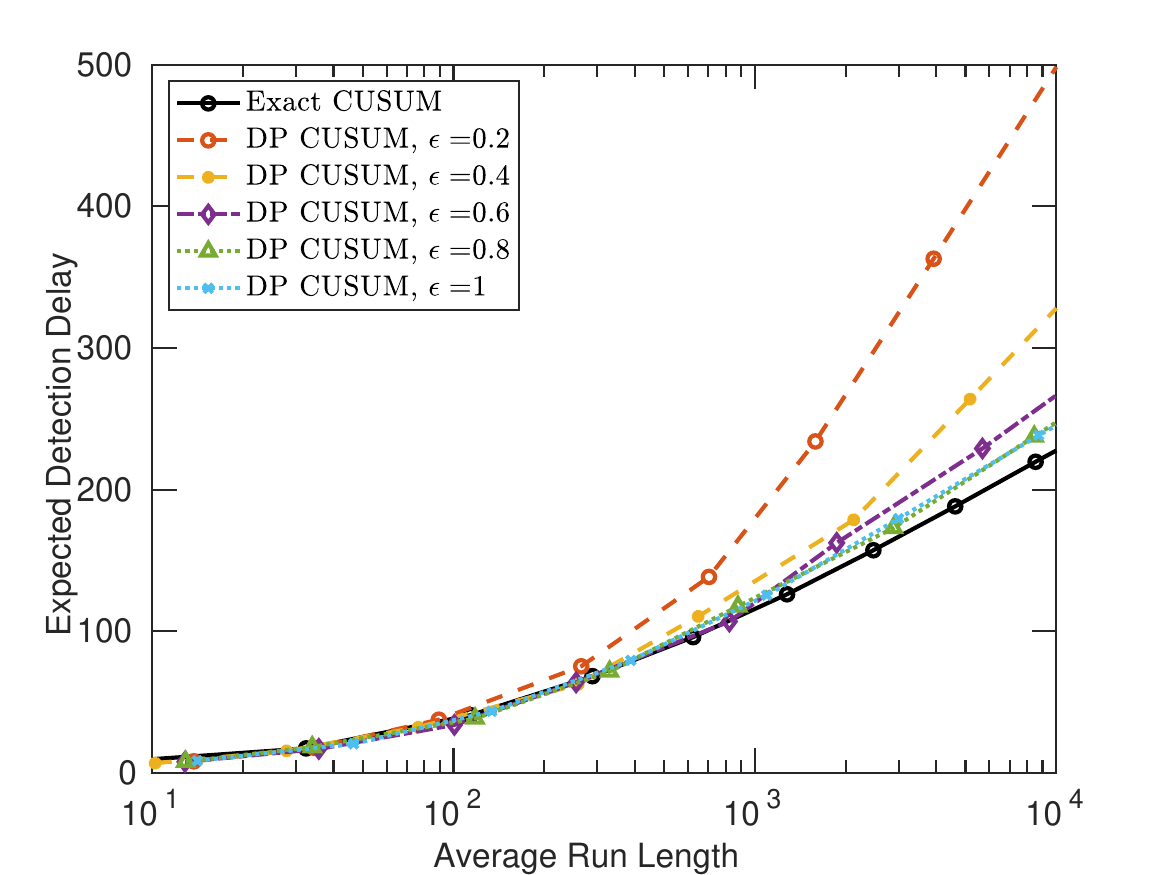}   & \includegraphics[width=0.48\linewidth]{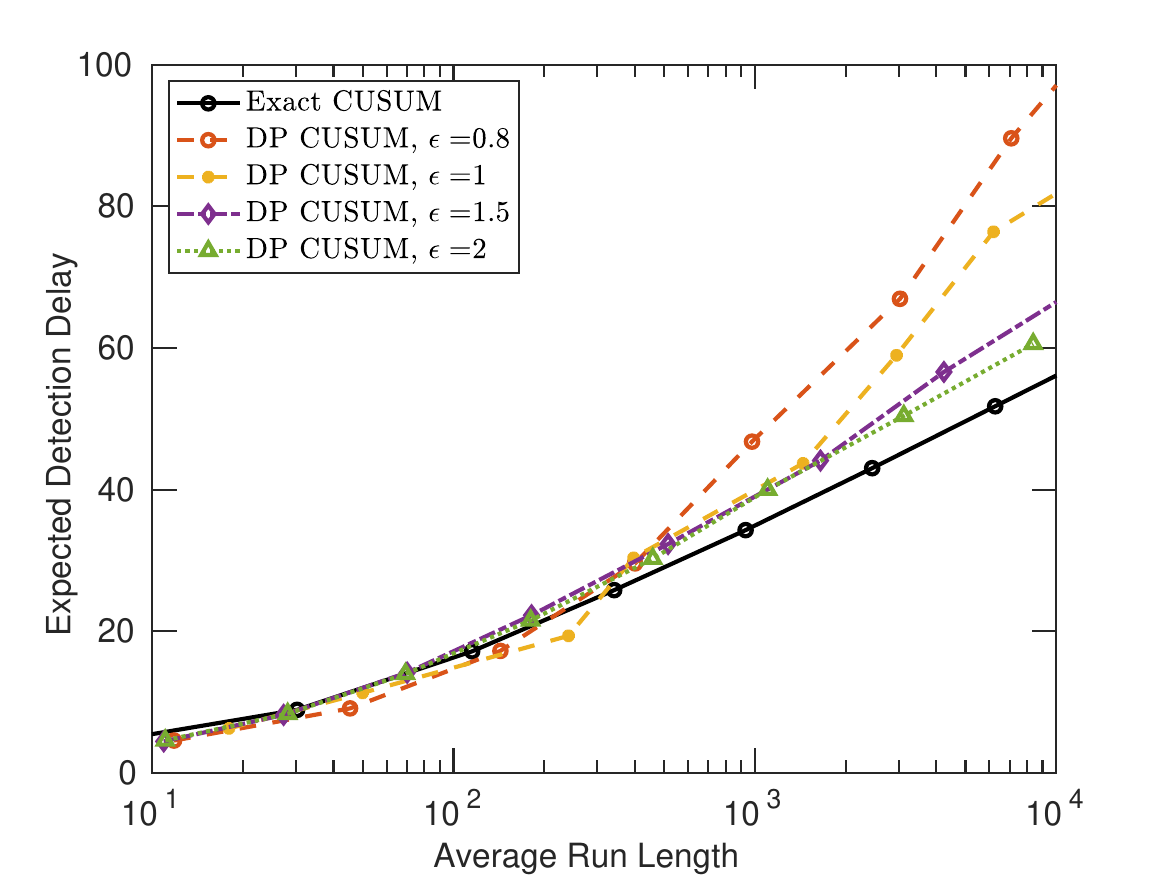} \\
       \small{(a) $\text{Lap}(0,1)\to\text{Lap}(0.2,1)$}  &  \small{(b) $\text{Lap}(0,1)\to\text{Lap}(0.5,1)$}
    \end{tabular}
    \caption{Average detection delay versus average run length of the DP-CUSUM procedure under Laplace distributions at various privacy levels $\epsilon$ for: (a) mean shift from 0 to 0.2; (b) mean shift from 0 to 0.5. The average run length and detection delay are averaged over 10,000 trials.}
    \label{fig:lap-edd}
\end{figure}

We further compare the performance of our DP-CUSUM procedure with the baseline method, online private change-point detector (OnlinePCPD) \cite{cummings2018differentially}, as shown in Fig.~\ref{fig:lap-edd-baseline}. OnlinePCPD builds on its offline Report Noisy Max framework and uses a fixed-size sliding window of width $w$ over the data stream. At each time step $t$, the algorithm considers a sliding window of recent samples ${X_{t-w+1},\ldots,X_t}$ and, for every candidate change-point $k\in\{t-w+1,\ldots,t\}$, computes the partial log-likelihood ratio $S_{k,t}=\sum_{i=k}^t\ell(X_i)$. The statistic used for detection is $\tilde S_t = (\max_{t-w+1\le k \le t} S_{k,t}) + Z_t$ with independent Laplace noise $Z_t \sim {\rm Lap}(8\Delta/\epsilon)$. The stopping rule is $\tilde T:=\inf\{t: \tilde S_t \ge b+W\}$, where $W \sim {\rm Lap}(4\Delta/\epsilon)$. This procedure is non-recursive, since all $S_{k,t}$ must be recomputed within each sliding window, resulting in an $O(w)$ per-step computational cost. While our proposed DP-CUSUM has only $O(1)$ per-step computational cost due to its recursive update of the detection statistics. Notably, in OnlinePCPD, Laplace noise $\text{Lap}(8\Delta/\epsilon)$ is added to the detection statistic and $\text{Lap}(4\Delta/\epsilon)$ to the threshold. In contrast, our method only adds $\text{Lap}(2\Delta/\epsilon)$ to both the statistic and the threshold, while still achieving the same level of $\epsilon$-DP, as guaranteed by Theorem~\ref{thm:dp}. 
 We adopt the same window size of 700 as used in \cite{cummings2018differentially}. Fig.~\ref{fig:lap-edd-baseline} shows that the detection delay of OnlinePCPD is substantially larger than that of our DP-CUSUM procedure across all privacy regimes (even when $\epsilon\le 2\Delta$), particularly when the ARL is large, under both mean shift scenarios $\mu = 0.2$ and $\mu = 0.5$ in the Laplace distribution setting.

\begin{figure}[!t]
    \centering
    \begin{tabular}{cc}
      \includegraphics[width=0.48\linewidth]{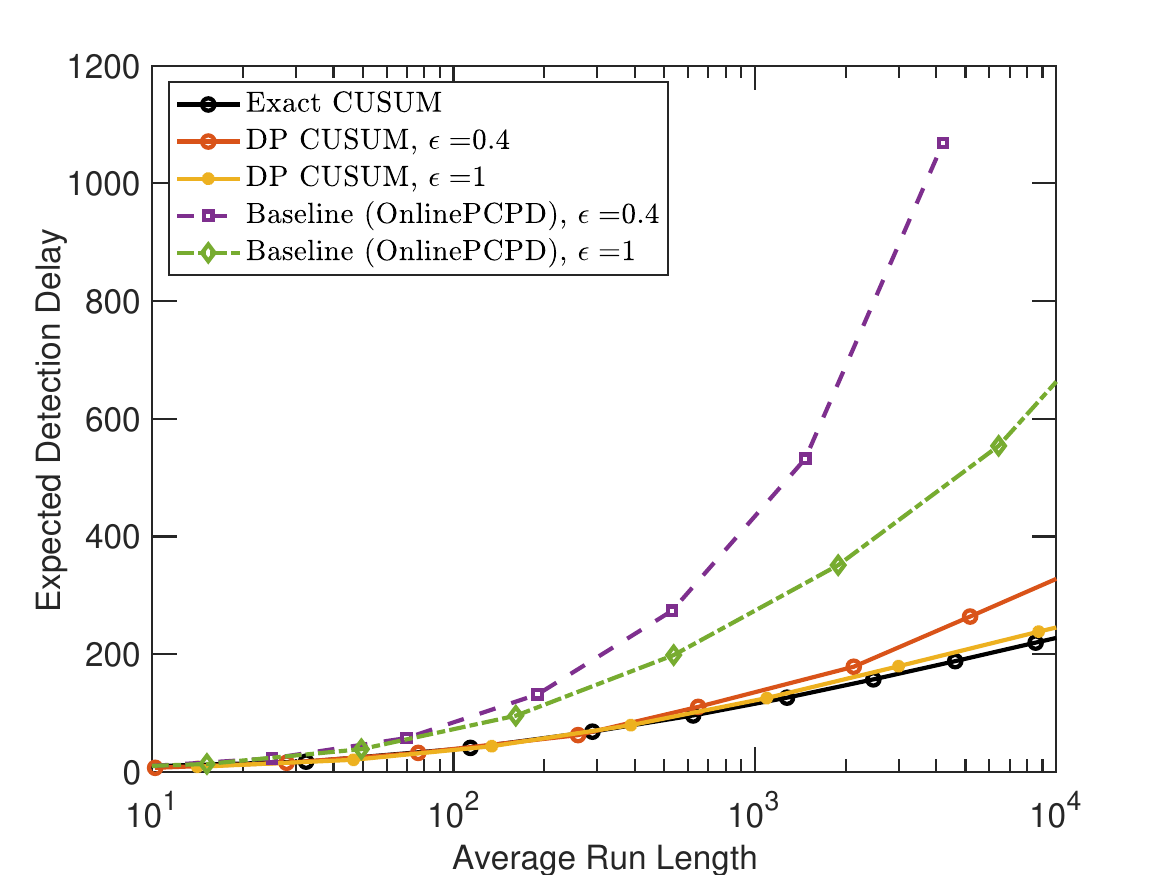}   & \includegraphics[width=0.48\linewidth]{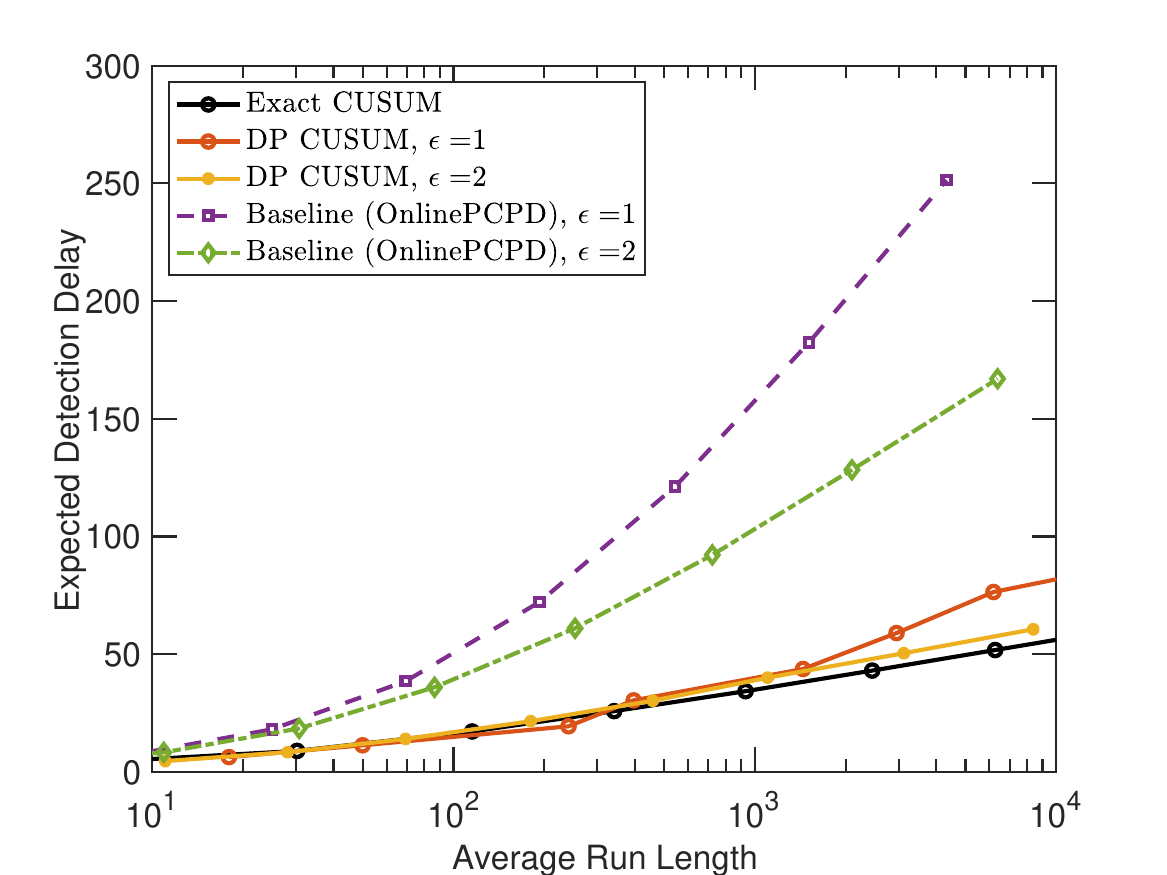} \\
       \small{(a) $\text{Lap}(0,1)$ $\to$ $\text{Lap}(0.2,1)$}  &  \small{(b) $\text{Lap}(0,1)$ $\to$ $\text{Lap}(0.5,1)$} 
    \end{tabular}
    \caption{Comparison of average detection delay of the DP-CUSUM procedure and the baseline method OnlinePCPD \cite{cummings2018differentially} under Laplace distributions with different privacy levels for: (a) mean shift from 0 to 0.2; (b) mean shift from 0 to 0.5. The average run length and detection delay are averaged over 10,000 trials. Window size in OnlinePCPD is set as 700 according to \cite{cummings2018differentially}.}
    \label{fig:lap-edd-baseline}
\end{figure}

\subsection{\algname{} with Unbounded LLR: Normal Distribution}\label{num:normal}

We consider a mean shift in the Normal distribution family $N(\mu,1)$, which has unbounded log-likelihood ratios. 
We set the pre-change distribution as $N(0, 1)$ and the post-change distribution as $N(\mu, 1)$. The LLR for an observation $x$ is given by $\ell(x) =  \mu x - \mu^2/2.$ We set $A_\delta = 2\mu \cdot z_{\delta/4} + \mu^2$, and set $\delta=0.1$ across all experiments. We simulate the detection delay and average run length of our \algname{} procedure and the exact CUSUM procedure under two $\mu$ values and varying $\epsilon$ values. Each configuration is repeated 10,000 times to compute the average performance. We consider two different mean shift magnitudes $\mu=0.1$ and $\mu=0.5$. Under each mean value, we consider varying privacy parameters. Specifically, for $\mu=0.1$, we consider $\epsilon\in\{0.5,1,1.5\}$ and for $\mu=0.5$, we consider $\epsilon\in\{0.5,2,4\}$. Again, these $\epsilon$ values are chosen to ensure coverage around the critical threshold $\epsilon = 2A_\delta$ for asymptotic optimality. We plot the results in Fig.~\ref{fig:normal-edd}.

For $\mu = 0.1$, we have $A_\delta = 0.402$. By Eq.~(\ref{eq:wadd-asym2}), the \algname{} procedure is guaranteed to be asymptotically optimal when $\epsilon \ge 2 A_\delta=0.804$, which aligns with the results in Fig.~\ref{fig:normal-edd}(a) where the detection delays for $\epsilon = 1$ and $1.5$ are close to that of the exact CUSUM. A similar pattern is observed for $\mu = 0.5$, where $A_\delta = 2.21$. In Fig.~\ref{fig:normal-edd}(b), the delay for $\epsilon = 4$ is close to that of the exact CUSUM, while a smaller $\epsilon$ leads to worse performance. It is also worth noting that performance degradation is less pronounced for smaller $\mu$ under the same $\epsilon$, likely because the detection problem is already difficult for smaller $\mu$ values---leading to larger delays even for the exact CUSUM procedure.

We again compare the performance of our \algname{} procedure with the baseline method OnlinePCPD in the Normal distribution setting (Fig.~\ref{fig:normal-edd-baseline}). For unbounded log-likelihood ratios, OnlinePCPD uses the same $A_\delta$ from \eqref{eq:Adelta} and adds Laplace noise $\text{Lap}(4A_\delta/\epsilon)$ and $\text{Lap}(8A_\delta/\epsilon)$ to the statistic and threshold, respectively. For a fair comparison, we set $\delta = 0.1$ for OnlinePCPD and use the same window size of 700 as in \cite{cummings2018differentially}. Fig.~\ref{fig:normal-edd-baseline} shows similar results: \algname{} achieves smaller detection delays than OnlinePCPD across all privacy parameters tested, especially when the ARL is large, under both $\mu = 0.1$ and $\mu = 0.5$.

\begin{figure}[!t]
    \centering
    \begin{tabular}{cc}
      \includegraphics[width=0.48\linewidth]{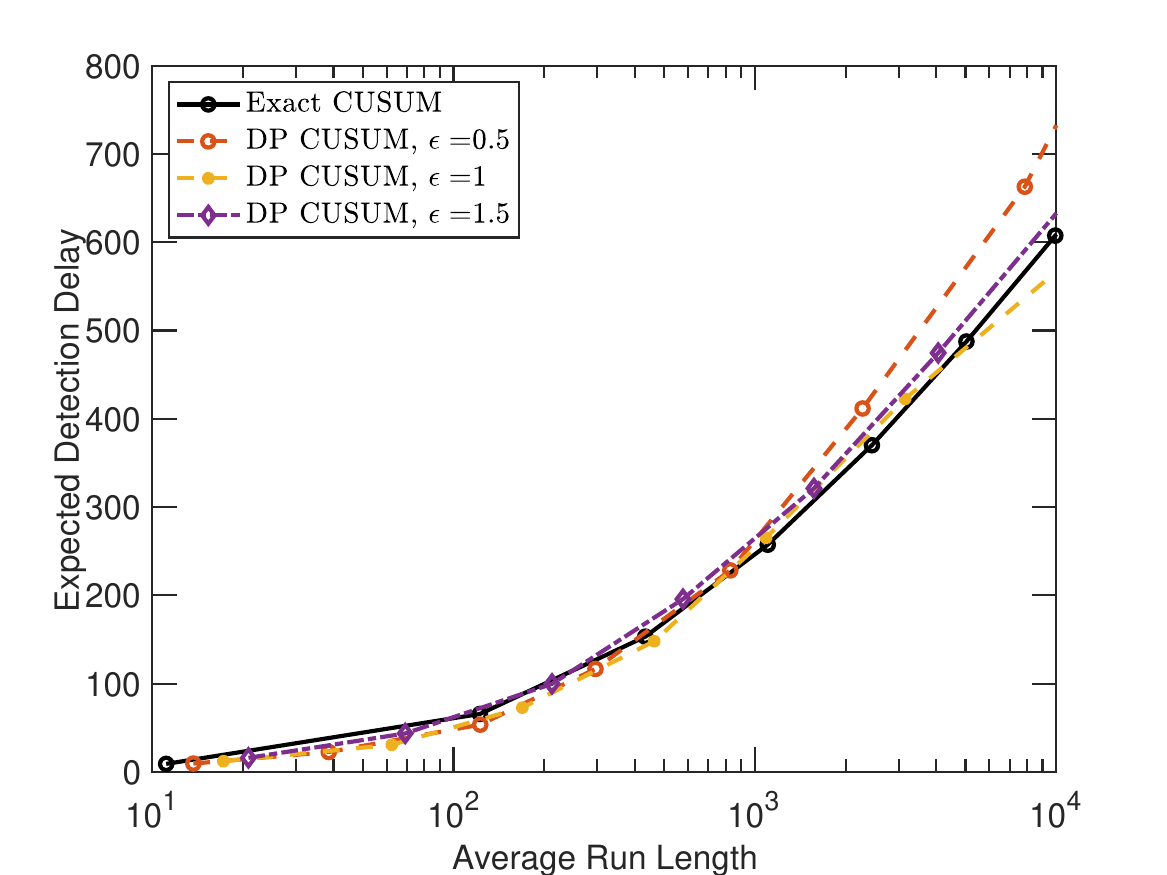}   & \includegraphics[width=0.48\linewidth]{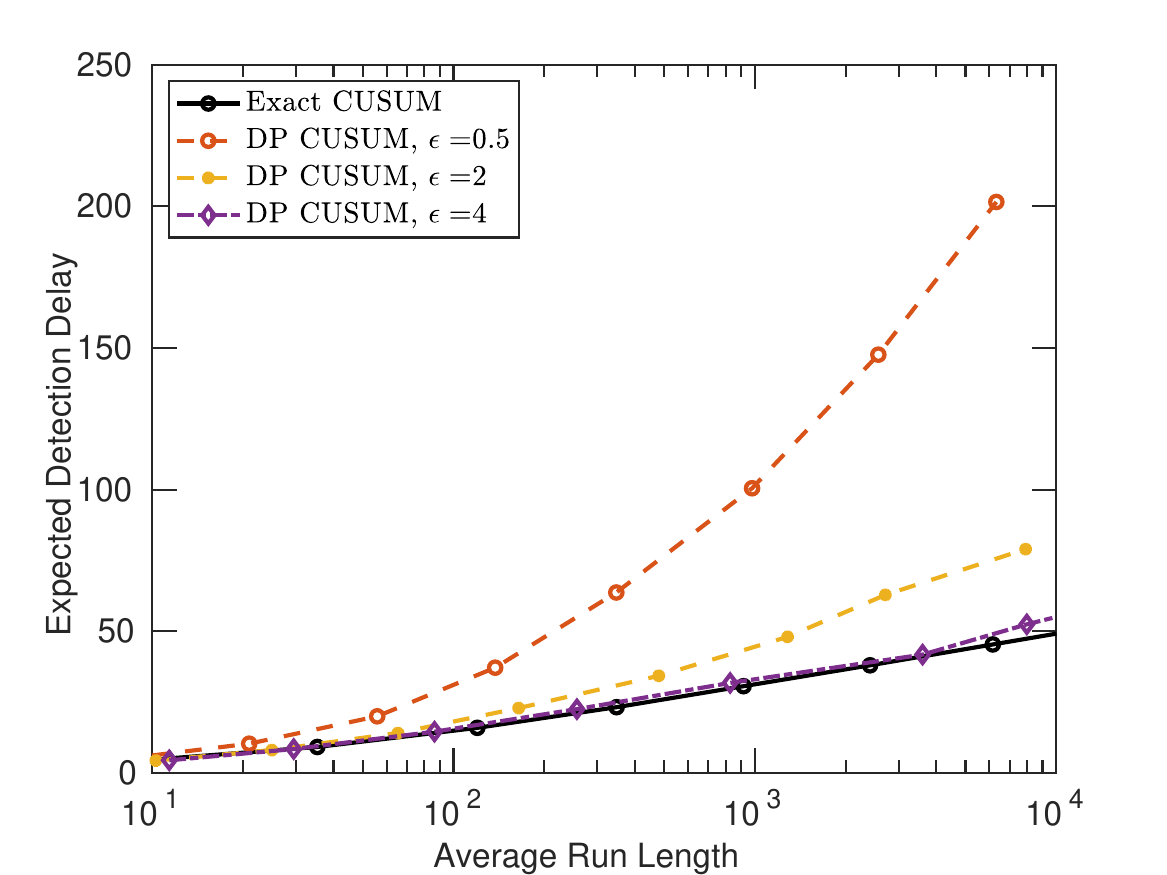} \\
       \small{(a) $N(0,1)$ $\to$ $N(0.1,1)$}   &  \small{(b) $N(0,1)$ $\to$ $N(0.5,1)$}
    \end{tabular}
    \caption{Average detection delay versus average run length of the DP-CUSUM procedure under Normal distributions at various privacy levels for: (a) mean shift from 0 to 0.1; (b) mean shift from 0 to 0.5. The average run length and detection delay are averaged over 10,000 trials.}
    \label{fig:normal-edd}
\end{figure}

\begin{figure}[!t]
    \centering
    \begin{tabular}{cc}
      \includegraphics[width=0.48\linewidth]{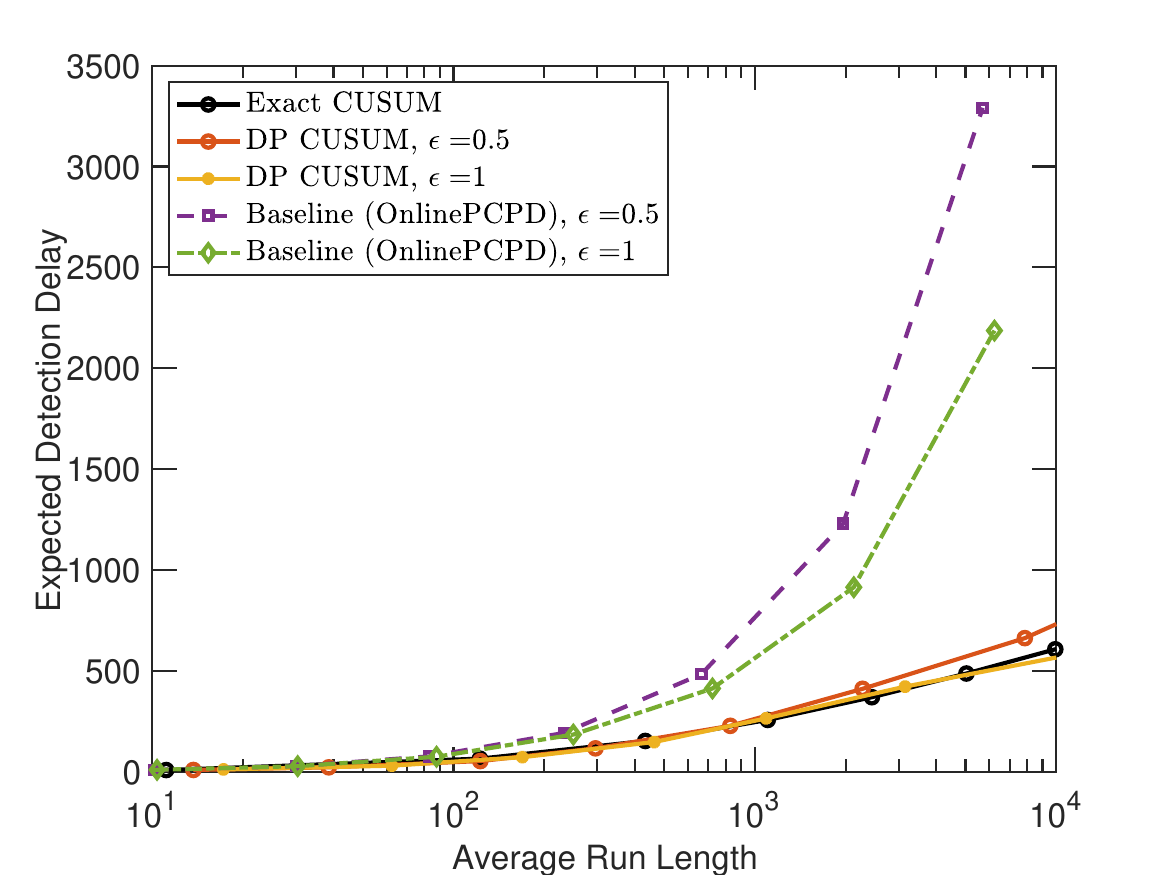}   & \includegraphics[width=0.48\linewidth]{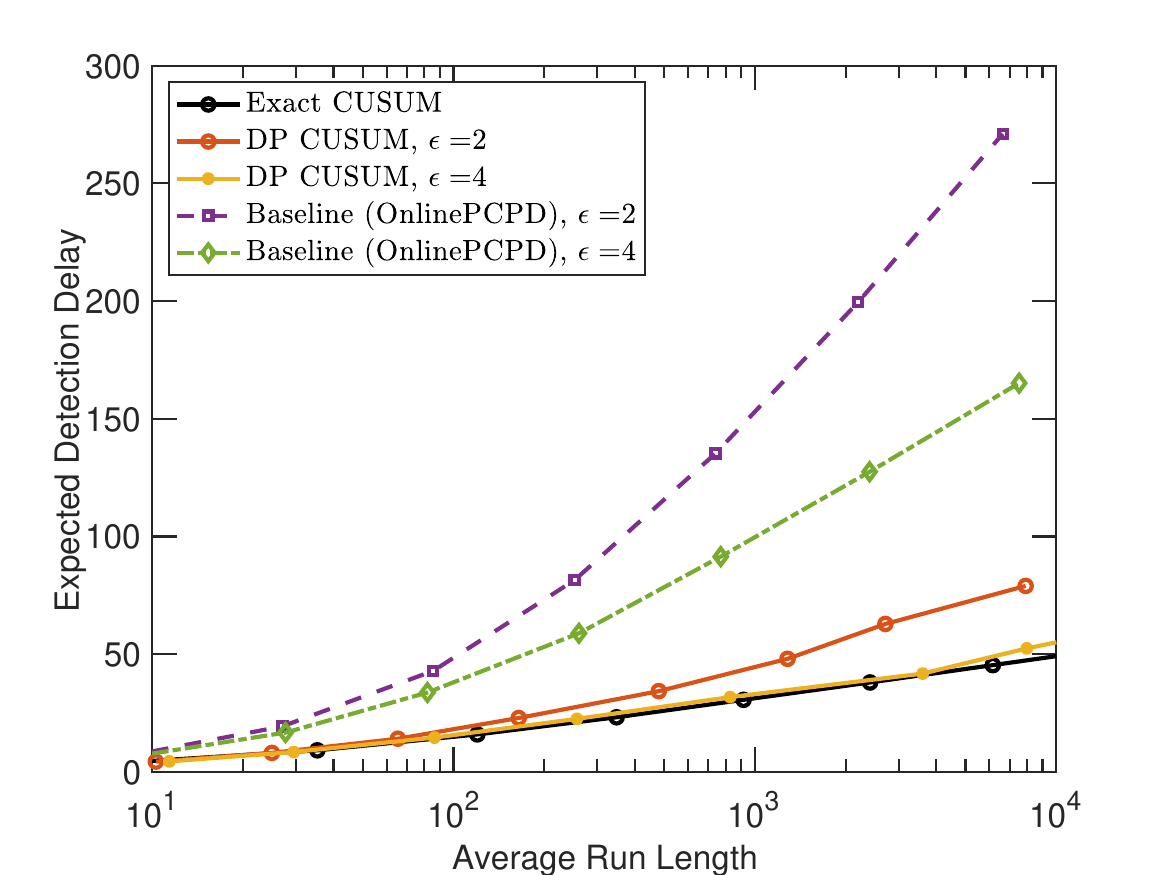} \\
       \small{(a) $N(0,1)$ $\to$ $N(0.1,1)$}   &  \small{(b) $N(0,1)$ $\to$ $N(0.5,1)$} 
    \end{tabular}
    \caption{Comparison of average detection delay of the DP-CUSUM procedure and the baseline method \cite{cummings2018differentially} under Normal distributions with different privacy levels for: (a) mean shift from 0 to 0.1; (b) mean shift from 0 to 0.5. The average run length and detection delay are averaged over 10,000 trials. Window size in OnlinePCPD is set as 700 according to \cite{cummings2018differentially}.}
    \label{fig:normal-edd-baseline}
\end{figure}

\section{Conclusion}
We introduce a theoretically grounded DP-CUSUM procedure for sequential change detection with differential privacy constraints. We derive the lower bound to the average run length and the upper bound to the worst-case average detection delay of DP-CUSUM. The theoretical results also yield the first-order asymptotic optimality of DP-CUSUM under weak privacy requirements. Experimental results on synthetic data demonstrate the good performance of our DP-CUSUM procedure compared with the baseline method and demonstrate our theoretical findings. 
Our theoretical analysis provides analytical techniques for future extensions and the design of DP-CUSUM variants across diverse practical applications. Future work includes adapting this framework to more complex settings, such as unknown post-change distributions, multi-sensor detection with controlled sensing, network-structured data, and related applications.

\section*{Acknowledgments}
The authors would like to thank the two anonymous reviewers and the Associate Editor for their insightful comments and constructive suggestions, which have greatly improved the clarity and quality of this work.

\appendices

\section{Useful Lemmas}\label{app:lemmas}

\begin{lemma}\label{lem:sensi}
For a pair of neighboring sequence $X_{(1:t)}$ and $X'_{(1:t)}$ that are different in one element only, denote $S_t$ and $S'_t$ as the CUSUM statistics at time $t$ based on the data sequence $X_{(1:t)}$ and $X'_{(1:t)}$, respectively. Let $\Delta_t=\max_{X_{(1:t)},X'_{(1:t)}}|S_t-S'_t|$. Then we have $\Delta_t\le \Delta$ for any $t\ge 1$.    
\end{lemma}
\begin{proof}
 If $X_{(1:t)}$ and $X'_{(1:t)}$ are different at time $t,$ i.e., $X_1=X'_1, \ldots, X_{t-1}=X'_{t-1},$ but $X_t\neq X'_t,$ then we have $|S_t-S'_t|\le \Delta$ for any pair of $X_t\neq X'_t$ by Definition~\ref{def:sensi}. If $X_{(1:t)}$ and $X'_{(1:t)}$ are different at time before $t$, then we have $X_t=X'_t$ and by definition of CUSUM statistics,
 \begin{align*}
 & |S_t-S'_t| \\  & =  |\max\left(0,S_{t-1}\right) + \ell(X_t) - \max\left(0,S'_{t-1}\right) - \ell(X_t)| \\
  &\le  |S_{t-1}-S'_{t-1}| \le  \Delta_{t-1}.   
 \end{align*}
 Therefore, we have $\Delta_t\le \max (\Delta, \Delta_{t-1}).$ Moreover, it is easy to see $\Delta_1=\Delta,$ which implies $\Delta_t\le \Delta$ for any $t\ge 1$.    
\end{proof}

\begin{lemma}\label{lem:mei}\cite{mei2005information,Siegmund1985}
For the CUSUM statistics $S_t$ in \eqref{eq:1}, under the pre-change measure we have 
\begin{align*}
\Pro_\infty(S_t \ge b) \le e^{-b}, \quad \forall t\in\mathbb{N}, \, \forall b \in \mathbb{R}^{+}. 
\end{align*}
\end{lemma}
\begin{proof} The proof can be found in \cite[Lemma 3]{mei2005information} and we omit here. 
\end{proof}

\begin{corollary}\label{cor:cusumMGF}
For the CUSUM statistics $S_t$ in \eqref{eq:1}, under the pre-change measure we have
\begin{equation}\label{eq:cusumMGF}
    \Exp_\infty[e^{\lambda S_t}] \le \frac{1}{1-\lambda}, \quad \forall t\in\mathbb{N}, \, \forall \lambda \in (0,1).
\end{equation}
\end{corollary}
\begin{proof}
From Lemma \ref{lem:mei}, $\Pro_\infty(S_t \ge b) \le e^{-b}$, we then have for any $t=1,2,\ldots,$ and any $0 < \lambda < 1$, 
\begin{align*}
& \Exp_\infty[e^{\lambda S_t}] \\ 
& =  \int_0^\infty \Pro_\infty(e^{\lambda S_t}\ge y)dy  = \int_0^\infty \Pro_\infty( S_t\ge \frac{\log y}{\lambda})dy \\
 &\overset{(i)}{=}  \int_{-\infty}^\infty \Pro_\infty( S_t\ge u) \lambda e^{\lambda u} du  \\
 & = \int_{-\infty}^0 \Pro_\infty( S_t\ge u) \lambda e^{\lambda u} du + \int_{0}^\infty \Pro_\infty( S_t\ge u) \lambda e^{\lambda u} du \\
& \overset{(ii)}{\le} \int_{-\infty}^0 \lambda e^{\lambda u} du + \int_{0}^\infty e^{-u} \lambda e^{\lambda u} du \\
& = 1 + \frac{\lambda}{1-\lambda} = \frac{1}{1-\lambda},
\end{align*}
where $(i)$ by change of variables $y=e^{\lambda u}$, and $(ii)$ by $\Pro_\infty( S_t\ge u) \le 1$ and  $\Pro_\infty(S_t \ge u) \le e^{-u}$.
   
\end{proof}

\section{Proofs of main theorems}\label{app:proofs}

\begin{proof}[Proof of Theorem \ref{thm:dp}]
We follow the proof of AboveThreshold in \cite{dwork2014algorithmic}. 
For any given $t\in\mathbb N$, we denote $\{S_k\}_{1\le k \le t}$ and $\{S'_k\}_{1\le k \le t}$ as the series of CUSUM statistics based on the data sequence $X_{(1:t)}$ and $X'_{(1:t)}$, respectively. Here, $X_{(1:t)}$ and $X'_{(1:t)}$ are neighboring sequences that are different in one element only. By definition \eqref{eq:dp-cusum-stat}, the DP-CUSUM statistics is $\{\tilde S_k = S_k + Z_k\}_{1\le k \le t}$ for the data sequence $X_{(1:t)}$ and $\{\tilde S'_k = S'_k + Z_k\}_{1\le k \le t}$ for the data sequence $X'_{(1:t)}$, where $Z_1,\ldots,Z_t$ are i.i.d. Laplace noise. The threshold is $b+W$ with Laplace noise $W$.

In the following, we fix the values of Laplace noise $Z_1,\dots,Z_{t-1}$ and only take probabilities over the randomness of the newly added noise $Z_t$ and threshold noise $W$. For any sequence $X_{(1:t)}$, we have:
\begin{equation}\label{eq:dp-proof-eq1}
\begin{aligned}
& \Pro_{Z_t,W}(\tilde{T} =t \mid X_{(1:t)}) \\
& \overset{(i)}{=} \Pro_{Z_t,W}(\max_{1\le j\le t-1}{\tilde{S_j}}<b+W,\tilde{S}_t\ge b+W\mid X_{(1:t)} )  \\
& \overset{(ii)}{=} \Pro_{Z_t,W}(\max_{1\le j\le t-1}{\tilde{S_j}}-b<W \leq {S}_t + Z_t -b\mid X_{(1:t)} )   \\
& \overset{(iii)}{=} \int_{-\infty}^\infty  \int_{-\infty}^{\infty} f_{Z_t}(z) f_W(w) \\
& \hspace{50pt} \cdot \1[w \in (\max_{1\le j\le t-1}{\tilde{S_j}}-b, {S}_t + z -b]]\, dw dz, 
\end{aligned} 
\end{equation}
where $(i)$ is by definition of stopping time $\tilde T = \tilde T(b)$ in \eqref{eq:dp-cusum-stop}, $(ii)$ is by the definition of private detection statistics $\tilde S_t = S_t + Z_t$, and $(iii)$ is by substituting the pdf $f_{Z_t}(\cdot)$ of $Z_t \sim \text{Lap}(\frac{2\Delta}{\epsilon})$ and pdf $f_{W}(\cdot)$ of $W \sim \text{Lap}(\frac{2\Delta}{\epsilon})$. 

For the neighboring sequence $X'_{(1:t)}$ that differs from $X_{(1:t)}$ in at most one entry, we make a change of variable as follows:
\begin{align*}
\hat z &= z + \max_{1\le j\le t-1}\tilde{S'_j} - \max_{1\le j\le t-1}\tilde{S_j} + S_t - S_t':=z+\delta_0, \\
\hat w &= w + \max_{1\le j\le t-1}\tilde{S'_j} - \max_{1\le j\le t-1}\tilde{S_j}:=w+\delta_1.
\end{align*}
Assume $X_{(1:t)}$ and $X'_{(1:t)}$ are different at time $k$ only, i.e., $X_1=X'_1, \ldots, X_{k-1}=X'_{k-1},X_{k+1}=X'_{k+1},\ldots,X_t = X'_t$, but $X_k\neq X'_k$. Then we have $S_1=S'_1, \ldots, S_{k-1}=S'_{k-1},$ but $S_k\neq S'_k$. Without loss of generality, we assume $S_k > S'_k$. 
By the definition of CUSUM statistics \eqref{eq:1}, we then have $S_{k+1} \ge S'_{k+1}$, $\ldots$, $S_t \ge S'_t$. 
Since we fixed $Z_1,\dots,Z_{t-1}$, we have $\tilde S_1 = \tilde S'_1,\ldots, \tilde S_{k-1} = \tilde S'_{k-1}; \tilde S_k > \tilde S'_k, \tilde S_{k+1} \ge \tilde S'_{k+1}, \ldots, \tilde S_{t-1} \ge \tilde S'_{t-1}$. Therefore, we have $\max_{1\le j\le t-1}\tilde{S'_j} - \max_{1\le j\le t-1}\tilde{S_j}\le 0$. Moreover, since $\tilde S_{k}\ge \tilde S'_k$ for $k=1,\ldots,t-1$, we have 
\begin{align*}
& |\max_{1\le j\le t-1}\tilde{S'_j} - \max_{1\le j\le t-1}\tilde{S_j}| \le \max_{1\le j\le t-1} |\tilde{S'_j} -\tilde{S_j}|\\
& = \max_{1\le j\le t-1} |{S'_j} - {S_j}| \le \Delta,    
\end{align*}
by Lemma \ref{lem:sensi} in Appendix \ref{app:lemmas}; thus $|\hat w - w| \leq \Delta$. Meanwhile, we have $S_t - S_t' \in [0,\Delta]$. Combining these together, we have $|\hat z - z|\leq \Delta$ as well. 

We can continue to Eq. \eqref{eq:dp-proof-eq1} as:
\begin{align*} 
& \Pro_{Z_t,W}(\tilde{T} =t \mid X^{1:t})  \\
&\overset{(iv)}{=} \int_{-\infty}^\infty  \int_{-\infty}^{\infty} f_{Z_t}(\hat z-\delta_0) f_W( \hat w-\delta_1) \\
& \hspace{30pt} \cdot \1[\hat w-\delta_1   \in (\max_{1\le j\le t-1}{\tilde{S_j}}-b, {S}_t + \hat z-\delta_0 -b]]\, d\hat w d\hat z \\
&\overset{(v)}{\le } \int_{-\infty}^\infty  \int_{-\infty}^{\infty} e^{\epsilon/2}f_{Z_t}(\hat z) e^{\epsilon/2} f_W(\hat w) \\
& \hspace{70pt} \cdot \1[\hat w \in (\max_{1\le j\le t-1}\tilde{S'_j}-b, S_t' +\hat z    -b]]\, d\hat w d\hat z \\
& \overset{(vi)}{=} e^{\epsilon}\Pro_{Z_t,W}(\tilde{T} =t \mid X'_{(1:t)} ).
\end{align*}
Here $(iv)$ is by change of variables $\hat z = z+\delta_0$ and $\hat w = w + \delta_1$, $(v)$ is due to $|\delta_0|\le \Delta$, $|\delta_1|\le \Delta$, and if $Y\sim \text{Lap}(\frac{2\Delta}{\epsilon})$, then $\frac{f_Y(y)}{f_Y(y')} = e^{-\frac{\epsilon}{2\Delta} (|y|-|y'|)}\leq e^{\frac{\epsilon}{2\Delta} |y-y'|}$, and $(vi)$ is due to the same derivation as in the derivation of Eq. \eqref{eq:dp-proof-eq1}. 
This concludes the proof that $\tilde{T}$ is $\epsilon$-DP.
\end{proof}

\begin{proof}[Proof of Theorem \ref{thm:arl}]
We fix the threshold $b>0$ and write $\tilde T = \tilde T(b)$ for simplicity. Condition on $W=w$ for some fixed $w \ge 0$, we first compute $\Exp_{\infty}[\tilde{T}|W=w]$ and then apply the law of total expectation to compute $\Exp_{\infty}[\tilde{T}]$. For any $x>0,$ and any $\lambda \in (0,1),$ by Chebyshev’s inequality, 
\begin{align*}
& \Exp_\infty[\tilde{T} | W=w] \\
& \ge x \Pro_\infty(\tilde{T}\ge x | W=w) \\
&=x(1-\Pro_\infty(\tilde{T}< x | W=w))\\
&=x\left(1-\Pro_\infty(S_n+Z_n\ge b+w, \,\text{for some $1\le n\le x$})\right)\\
&\ge x(1-\sum_{n=1}^{\lfloor x \rfloor}\Pro_\infty(S_n+Z_n\ge b+w))\\
& \ge x(1-\sum_{n=1}^{\lfloor x \rfloor} e^{-\lambda (b+w)}\Exp_\infty[e^{\lambda Z_n}]\Exp_\infty[e^{\lambda S_n}]  )  \\
& \ge x(1-x e^{-\lambda (b+w)}\Exp_\infty[e^{\lambda Z}] \frac{1}{1-\lambda }  ),
\end{align*}
where the last inequality due to $\Exp_\infty[e^{\lambda S_n}]\le \frac{1}{1-\lambda}$ (as shown in Corollary\,\ref{cor:cusumMGF}) and $\lfloor x \rfloor \le x$.

Note that for any $u>0,$ the function $x(1-xu)$ is maximized at $x=1/(2u)$ with the maximum value $1/(4u)$. Thus, taking $x=1/(2e^{-\lambda (b+w)}\Exp_\infty[e^{\lambda Z}] \frac{1}{1-\lambda })$ yields,
\begin{align*}
\Exp_\infty[\tilde{T}|W=w]&\ge \frac{1}{4}\cdot (1-\lambda) \cdot \frac{e^{\lambda (b+w)}}{\Exp[e^{\lambda Z}]}.
\end{align*}
Since $Z \sim \text{Lap}(\frac{2\Delta}{\epsilon})$, we have for any $0<\lambda< \frac{\epsilon}{2\Delta}$, $\Exp[e^{\lambda Z}]=\frac{1}{1-4\Delta^2\lambda^2/\epsilon^2}$. Substituting this into the expression above, we obtain
\begin{align*}
\Exp_\infty[\tilde{T}|W=w]&\ge \frac{e^{\lambda (b+w)}}{4}(1-\lambda)(1-(\frac{2\Delta}{\epsilon}\lambda)^2) \\
& \ge \frac{e^{\lambda (b+w)}}{4}(1-\lambda)(1- \frac{2\Delta}{\epsilon}\lambda),
\end{align*}
where the last inequality is due to $\lambda < \frac{\epsilon}{2\Delta}$.

We then consider the following two cases.
\begin{enumerate}
    \item If $\epsilon \le 2\Delta$, we have $1-\lambda \ge 1- \frac{2\Delta}{\epsilon}\lambda$ and thus
\begin{align*}
\Exp_\infty[\tilde{T}|W=w]&\ge \frac{e^{\lambda (b+w)}}{4}(1- \frac{2\Delta}{\epsilon}\lambda)^2.
\end{align*}
 Taking the logarithm and differentiating with respect to $\lambda$, we find the unique stationary point of the right-hand side is $\lambda^*=\frac{\epsilon}{2\Delta}-\frac{2}{b+w}.$ Moreover, if $w> 4\Delta/\epsilon,$ we have $\lambda^*\in(0,\frac{\epsilon}{2\Delta}),$ which implies $\lambda^*$ is the maximizer of the right-hand side. Thus  
\begin{align*}
\Exp_\infty[\tilde{T}|W=w]&\ge e^{\frac{\epsilon}{2\Delta}(b+w)-2}(\frac{2\Delta}{(b+w)\epsilon})^2.
\end{align*}
Then we have
\begin{align*}
\Exp_{\infty}[\tilde{T}] & = \int_{-\infty}^{\infty} \Exp_\infty[\tilde{T}|W=w]f_W(w) dw \\
& \ge \int_{4\Delta/\epsilon}^{\infty} e^{\frac{\epsilon}{2\Delta}(b+w)-2}(\frac{2\Delta}{(b+w)\epsilon})^2 f_W(w) dw \\
& = e^{\frac{\epsilon}{2\Delta} b - 2} (\frac{2\Delta}{\epsilon})^2 \int_{{4\Delta/\epsilon}}^{\infty} e^{\frac{\epsilon}{2\Delta} w} \frac{1}{(b+w)^2} \frac{\epsilon}{4\Delta} e^{-\frac{\epsilon}{2\Delta}w} dw \\
& = \frac{\Delta}{\epsilon (b+4\Delta/\epsilon)} \cdot e^{\frac{\epsilon}{2\Delta} b - 2} \ge \frac{e^{\frac{\epsilon}{2\Delta} b - 2}}{2b+4} \ge \frac{e^{\frac{\epsilon}{2\Delta} b - 2}}{4(b+1)^2}.
\end{align*}

\item If $\epsilon > 2\Delta$, we have 
\begin{align*}
\Exp_\infty[\tilde{T}|W=w]&\ge \frac{e^{\lambda (b+w)}}{4}(1-\lambda)^2.
\end{align*}
Similarly, the unique stationary point of the right-hand side is $\lambda^*=1-2/(b+w).$ Moreover, if $b>2,$ we have $\lambda^*\in(0,\frac{\epsilon}{2\Delta}),$ which implies that $\lambda^*=1-2/(b+w)$ is also the maximizer of the right-hand side. Thus, 
\begin{align*}
\Exp_\infty[\tilde{T}|W=w]&\ge e^{b+w-2}\frac{1}{(b+w)^2}.
\end{align*}
Then we have
\begin{align*}
\Exp_{\infty}[\tilde{T}] & = \int_{-\infty}^{\infty} \Exp_\infty[\tilde{T}|W=w]f_W(w) dw \\
& \ge \int_{0}^{\infty} e^{b+w-2}\frac{1}{(b+w)^2} f_W(w) dw \\
& = e^{b-2} \int_{0}^{\infty} e^{w} \frac{1}{(b+w)^2} \frac{\epsilon}{4\Delta} e^{-\frac{\epsilon}{2\Delta}w} dw  \\
& \ge  \frac{\epsilon}{4\Delta} e^{b-2} \int_{0}^{\infty}  \frac{1}{(b+w)^2} e^{-\frac{\epsilon}{2\Delta}w} dw \\
& \ge  \frac{\epsilon}{4\Delta} e^{b-2} \int_{0}^{1}  \frac{1}{(b+w)^2} e^{-\frac{\epsilon}{2\Delta}w} dw \\
& \ge \frac{\epsilon}{4\Delta} e^{b-2} 
 \frac{1}{(b+1)^2} \int_{0}^{1} e^{-\frac{\epsilon}{2\Delta}w} dw \\
& = \frac{\epsilon}{4\Delta} e^{b-2} 
 \frac{1}{(b+1)^2} \frac{2\Delta  }{\epsilon}(1- e^{-\frac{\epsilon}{2\Delta}})  \\
 & \ge \frac12 e^{b-2} 
 \frac{1}{(b+1)^2} (1- e^{-1}) \ge \frac14 \frac{e^{b-2}}{(b+1)^2}.
\end{align*}
\end{enumerate}
Combining these two cases together completes the proof.
\end{proof}

\begin{proof}[Proof of Lemma \ref{lem:wadd0}]
Consider the changepoint at $\tau$. Set statistics $U_1=\cdots=U_\tau=-\min_{x}\ell(x)<0$, and $U_{t} = \max(0,U_{t-1}) + \ell(X_t)$ for $t>\tau$, and let $\tilde T_0(b)$ be the stopping time 
\[
\tilde T_0(b) := \inf\{t: U_t + Z_t \geq b+W\},
\]
where $Z_t\sim \text{Lap}(\frac{2\Delta}{\epsilon}), W\sim \text{Lap}(\frac{2\Delta}{\epsilon})$ are i.i.d. Laplace noise, the same as in the DP-CUSUM procedure. Then by definition we have $U_t \leq S_t$, $\forall t\in\mathbb{N}$, thus $\tilde T_0(b) \geq \tilde T(b)$. Moreover, we have
\begin{align*}
&\Exp_\tau[(\tilde T(b)-\tau)^{+}|\cF_\tau] \\ 
& = \Exp_\tau[(\tilde T(b)-\tau)\1(\tilde T(b)>\tau)|\cF_\tau] \\
& = \Exp_{Z_1,\ldots,Z_\tau,W} \Exp_\tau[(\tilde T(b)-\tau)\1(\tilde T(b)>\tau)|\cF_\tau, Z_1,...,Z_\tau,W] \\
& \overset{(i)}{\leq} \! \Exp_{Z_1,...,Z_\tau,W} \Exp_\tau[(\tilde T_0(b) \!-\!\tau)\1(\tilde T_0(b)\!>\!\tau)|\cF_\tau, Z_1,...,Z_\tau,W]  \\
& \overset{(ii)}{=} \! \Exp_{Z_1,\ldots,Z_\tau,W} \Exp_\tau[(\tilde T_0(b)-\tau)\1(\tilde T_0(b)>\tau)| Z_1,...,Z_\tau,W]  \\
& \overset{(iii)}{=} \Pro_{Z_1,\ldots,Z_\tau,W}(\tilde T_0(b)>\tau) \cdot \Exp_\tau[(\tilde T_0(b)-\tau)| \tilde T_0(b)>\tau] \\ 
& \le \Exp_\tau[\tilde T_0(b)-\tau| \tilde T_0(b)>\tau] \\
& = \Exp_0[\tilde T(b)].
\end{align*}
The inequality $(i)$ is due to $\tilde T_0(b) \geq \tilde T(b)$; the equality $(ii)$ is true because the stopping time $\tilde T_0$ does not employ any information from $\cF_\tau$, consequently it is independent from $\cF_\tau$; the equality $(iii)$ is due to the fact that $\1(\tilde T_0(b)>\tau)\in \sigma\{Z_1,\cdots, Z_{\tau}, W\}.$ The last equality is true because when $\tilde T_0(b) > \tau$,  statistically this is the same as starting at 0 with the change occurring at 0.
\end{proof}

\begin{proof}[Proof of Theorem \ref{thm:edd}]
Let the changepoint be $\tau=0$, i.e., all samples are drawn from the post-change distribution. Condition on $W=w$ for any $w \ge 0$, we denote $\tilde b := b+w>0$, define an alternative process $U_t = \sum_{i=1}^t \ell(X_i),t=1,2,\ldots$, with $U_0=0$, and define another stopping time $\nu$ by
\begin{align}\label{eq:new-stop}
\nu(\tilde{b})=\inf\{t\ge 1:U_t +Z_t\geq \tilde b\}.
\end{align}
Here $Z_t\sim \text{Lap}(\frac{2\Delta}{\epsilon})$ is the same Laplace noise as in \eqref{eq:dp-cusum-stat}.
Obviously we have $U_t \leq S_t$, thus $U_t+Z_t \leq S_t+Z_t=\tilde S_t$ and $\Exp_0[\tilde T(b)|W=w]\leq \Exp_0[\nu(\tilde{b})]$. In the following, we derive an upper bound for $\Exp_0[\nu(\tilde{b})]$. Note that for notational simplicity, we omit the conditioning on $W=w$ in the following proofs unless explicitly stated otherwise.

\noindent Step 1: We first show $\Exp_0[\nu(\tilde{b})] <\infty$ for any fixed $\tilde{b}>0$. Note that for any integer $t\ge1$ and $\lambda \in (0,\frac{\epsilon}{2\sqrt{2}\Delta}]$ we have
\begin{equation*}
\begin{aligned}
\Pro_0(\nu(\tilde{b}) \ge t+1)  & \le \Pro_0(U_t + Z_t < \tilde b ) \le \frac{\Exp_0[e^{-\lambda (U_t + Z_t)}]}{e^{-\lambda \tilde b}} \\
& = e^{\lambda \tilde b} \Exp_0[e^{-\lambda Z_t}] \Exp_0[e^{-\lambda (\sum_{i=1}^t \ell(X_i))}] \\
& = e^{\lambda \tilde b} \frac{1}{1-4\Delta^2\lambda^2/\epsilon^2} (\Exp_0[e^{-\lambda \ell(X_1)}])^t \\
&\le 2 e^{\lambda \tilde b} (\Exp_0[e^{-\lambda \ell(X_1)}])^t,    
\end{aligned}
\end{equation*}
where the last inequality is due to $1-4\Delta^2\lambda^2/\epsilon^2\ge \frac12$ when $\lambda \le \frac{\epsilon}{2\sqrt{2}\Delta}$. 
Recall that $\ell(x)$ is bounded with range $\Delta$, and $\Exp_0[\ell(X_1)]=I_0 >0$, thus by Hoeffding's inequality we have $
\Exp_0[e^{-\lambda \ell(X_1)}] \le e^{-\lambda I_0 + \frac{\lambda^2\Delta^2}{8}}$.
Substitute into the previous inequality, we have
\begin{align*}
\Pro_0(\nu(\tilde{b}) \ge t+1) & \le 2 e^{\lambda \tilde b} e^{-t (\lambda I_0 -  \frac{\lambda^2\Delta^2}{8})}. 
\end{align*}
Now take $\lambda =\min\{ \frac{4I_0}{\Delta^2}, \frac{\epsilon}{2\sqrt{2}\Delta}\}$, we have $C_1:=\lambda I_0 - \frac{\lambda^2\Delta^2}{8}>0$, and $\Pro_0(\nu(\tilde{b}) \ge t+1) \le 2 e^{\lambda \tilde b} e^{-t C_1}$. Thus we have $\Exp_0[\nu(\tilde{b})]=\sum_{t= 1}^{\infty}\Pro(\nu(\tilde{b})\ge t)<\infty$. 

\noindent Step 2: We derive an upper bound to the expected stopping time $\Exp_0[\nu(\tilde{b})]$ using Wald's equation. Since $\ell(X_1),\ell(X_2),\ldots$ are i.i.d. with $\Exp_0[\ell(X_1)]=I_0$, $U_t = \sum_{i=1}^t \ell(X_i)$, and $\nu(\tilde{b})$ is a stopping time with $\Exp_0[\nu(\tilde{b})]<\infty$ (as proved in Step 1), by Wald's equation \cite[Theorem 4.8.6]{durrett2019probability}, we have $\Exp_0[U_{\nu(\tilde{b})}] = \Exp_0[\nu(\tilde{b})] \Exp_0[\ell(X_1)]=\Exp_0[\nu(\tilde{b})]I_0$. For simplicity, we use $\nu$ to denote the stopping time $\nu(\tilde{b}).$ Then, we rewrite it as follows,
\begin{equation}\label{eq:edd-wald}
\begin{aligned}
\Exp_0[\nu] & = \frac{\Exp_0[U_{\nu}]}{I_0} =  \frac{\Exp_0[U_{\nu} + Z_{\nu} - \tilde b] + \Exp_0[\tilde b - Z_{\nu}]}{I_0} \\
& = \frac{\tilde b + \Exp_0[U_{\nu} + Z_{\nu} - \tilde b] + \Exp_0[ - Z_{\nu}]}{I_0},        
\end{aligned}
\end{equation}
where in the numerator, the first expectation is for the so-called overshoot $U_{\nu} + Z_{\nu} - \tilde b \ge 0$, and the second expectation is for the added Laplace noise $Z_\nu$ at the stopping time.

We assume $Z_0=0$. Note that by definition of $\nu$, $U_{\nu-1} + Z_{\nu-1} < \tilde b$, thus the overshoot is upper-bounded as follows
\begin{align*}
U_{\nu} + Z_{\nu}-\tilde b & = U_{\nu-1}+\ell(X_{\nu})  +Z_{\nu} - \tilde b +Z_{\nu-1}-Z_{\nu-1} \\
& =U_{\nu-1} + Z_{\nu-1}-\tilde b+\ell(X_{\nu})  +Z_{\nu}-Z_{\nu-1} \\
& \le \Delta+Z_{\nu} -Z_{\nu-1}.
\end{align*}
Substitute into Eq.~\eqref{eq:edd-wald}, we have 
\begin{equation}\label{eq:nu-upper}
\begin{aligned}
\Exp_0[\nu] & \le \frac{\tilde b + \Delta+ \Exp_0[Z_\nu - Z_{\nu-1}] +\Exp_0[-Z_\nu]}{I_0} \\
& = \frac{\tilde b + \Delta+ \Exp_0[-Z_{\nu-1}]}{I_0} \le \frac{\tilde b + \Delta+ \Exp_0[|Z_{\nu-1}|]}{I_0}.    
\end{aligned}    
\end{equation}

\noindent Step 3: We now derive an upper bound for $\Exp_0[|Z_{\nu-1}|]$. 
Let $m_1=\lfloor \frac{2\tilde b}{I_0}\rfloor+1$. We have
\begin{equation}\label{eq:all}
\begin{aligned}
& \Exp_0[|Z_{\nu-1}|]= \sum_{i=1}^{\infty}\Exp_0|Z_i \1(\nu=i+1)|\\
&\le \underbrace{\sum_{i=1}^{m_1}\Exp_0|Z_i \1(\nu=i+1)|}_{\text{Part I}} \\
& \hspace{20pt} +\underbrace{\sum_{i=m_1+1}^{\infty}\Exp_0\left(|Z_i| \1(\sum_{j=1}^{i}\ell(X_j)+Z_{i} < \tilde b)\right)}_{\text{Part II}}. 
\end{aligned}
\end{equation}

For Part I, note 
\begin{align}
\text{Part I}&\le \sum_{i=1}^{m_1}(\Exp_0|Z_i|^2)^{1/2}\sqrt{\Pro_0(\nu=i+1)} \notag\\
&\le 2\sqrt{2}(\frac{\Delta}{\epsilon})\left(\sum_{i=1}^{m_1}\sqrt{\Pro_0(\nu=i+1)}\right) \notag \\
& \le 2\sqrt{2}(\frac{\Delta}{\epsilon})\sqrt{(\sum_{i=1}^{m_1} 1) (\sum_{i=1}^{m_1} \Pro_0(\nu=i+1))} \notag \\
& \le 2\sqrt{2}(\frac{\Delta}{\epsilon})\sqrt{m_1} \notag \\
& \le 2\sqrt{2}(\frac{\Delta}{\epsilon})\sqrt{\frac{2\tilde b}{I_0}+1} \le 2\sqrt{2}(\frac{\Delta}{\epsilon})(\sqrt{\frac{2\tilde b}{I_0}} + 1) \notag \\
& = 4\frac{\Delta}{\epsilon} \sqrt{\frac{\tilde b}{I_0}} +2\sqrt{2}\frac{\Delta}{\epsilon} .\label{eq:part3}
\end{align}

For Part II, by Cauchy-Schwarz inequality, and for any $\lambda \in (0,\frac{\epsilon}{2
\sqrt{2}\Delta}]$, we have
\begin{align*}
&\Exp_0\left(|Z_{i}|\1(\sum_{j=1}^{i}\ell(X_j)+Z_{i}< \tilde b)\right) \\
& \le \left(\Exp_0|Z_{i}|^2\right)^{1/2}\sqrt{\Pro_0(\sum_{j=1}^{i}\ell(X_j)+Z_{i}< \tilde b)}\\
& \le 2\sqrt{2}(\frac{\Delta}{\epsilon}) \sqrt{\frac{\Exp_0 e^{-\lambda\left(\sum_{j=1}^{i}\ell(X_j)+Z_{i}\right)}}{e^{-\lambda \tilde b}}} \\
& \le 2\sqrt{2}(\frac{\Delta}{\epsilon}) \sqrt{\frac{e^{(\frac{\lambda^2\Delta^2}{8} - I_0 \lambda)i+\lambda \tilde b}}{1-4\Delta^2\lambda^2/\epsilon^2}} \le 4(\frac{\Delta}{\epsilon}) e^{\frac12\lambda \tilde b}e^{-\frac12(I_0 \lambda - \frac{\lambda^2\Delta^2}{8})i},
\end{align*}
where the last inequality is due to $1-4\Delta^2\lambda^2/\epsilon^2\ge \frac12$ when $\lambda \le \frac{\epsilon}{2\sqrt{2}\Delta}$. Taking $\lambda=\bar \lambda=\min\{4I_0/\Delta^2, \epsilon/(2\sqrt{2}\Delta)\}$ guarantees $(I_0\bar \lambda-\frac{{\bar \lambda}^2 \Delta^2}{8})>0$ and thus, 
\begin{align}
\text{Part II} & 
\le \frac{4\Delta}{\epsilon} e^{\frac12\bar \lambda \tilde b} \sum_{i=m_1+1}^\infty e^{-\frac12(I_0 \bar\lambda - \frac{{\bar\lambda}^2\Delta^2}{8})i} \notag \\
&
\overset{(1)}{\le } \frac{4\Delta}{\epsilon} e^{\frac12\bar \lambda \tilde b} \frac{1}{\frac12(I_0 \bar\lambda - \frac{{\bar\lambda}^2\Delta^2}{8})} e^{-\frac12m_1(I_0 \bar\lambda - \frac{{\bar\lambda}^2\Delta^2}{8})} \notag \\
& \overset{(2)}{\le} 
\frac{8\Delta}{\epsilon(I_0 \bar\lambda - \frac{{\bar\lambda}^2\Delta^2}{8})} 
e^{-\frac12\tilde b \bar\lambda (1 - \frac{{\bar\lambda}\Delta^2}{4I_0})} \le \frac{8\Delta}{\epsilon(I_0 \bar\lambda - \frac{{\bar\lambda}^2\Delta^2}{8})}, \label{eq:part2}
\end{align}
where $(1)$ is due to $\sum_{i=m_1+1}^{\infty} e^{-Ci} = \frac{e^{-C(m_1+1)}}{1-e^{-C}} \le \frac{e^{-C(m_1+1)}}{C/(1+C)} \le  \frac{e^{-Cm_1}}{C}$ for any $C>0$, and $(2)$ is due to $m_1 \ge 2\tilde b/I_0$.

\noindent Step 4: Finally, we substitute the upper bound for $\Exp_0[|Z_{\nu-1}|]$ as derived in Step 3 into Eq.~\eqref{eq:nu-upper} to complete the proof. Substituting the upper bounds in \eqref{eq:part3} and \eqref{eq:part2} into Eq.~\eqref{eq:all} and \eqref{eq:nu-upper}, we have for any $\tilde{b}>0$,
\begin{equation*}\label{eq:re}
\begin{aligned}
\Exp_0[\nu(\tilde{b})] & \leq \frac{\tilde b  + 4\frac{\Delta}{\epsilon} \sqrt{\frac{\tilde b}{I_0}} +2\sqrt{2}\frac{\Delta}{\epsilon}  + \frac{8\Delta}{\epsilon(I_0 \bar\lambda - \frac{{\bar\lambda}^2\Delta^2}{8})}}{I_0} \\
& = \frac{\tilde b  +  {4\frac{\Delta}{\sqrt{I_0}\epsilon}} \sqrt{\tilde b} + C}{I_0},    
\end{aligned}
\end{equation*}
where $C$ is a constant that depends only on $\epsilon$, $\Delta$, and $I_0$, but is independent of $\tilde{b}$.

Above we have shown the upper bound for $\Exp_0[\nu(b+w)|W=w]$ for any $b\ge 0$ when $w\ge 0$. Note that when $w<0$, we have $\Exp_0[\nu(b+w)|W=w] \le \Exp_0[\nu(b+|w|)|W=|w|]$. Using $\sqrt{b+|w|} \leq \sqrt{b} + \sqrt{|w|}$,  we have
\begin{align*}
    \Exp_0[\tilde T(b)] & = \Exp_W\Exp_0 [\nu(b+w)| W=w] \\
    & \leq \Exp_W\Exp_0 [\nu(b+|w|)| W=w] \\
    &\le\Exp_W [\frac{b+ |W| + 4\frac{\Delta}{\sqrt{I_0}\epsilon} (\sqrt{b} + \sqrt{|W|})+C}{I_0}] \\
    & = \frac{b}{I_0} +  \frac{\Exp_W|W|}{I_0} + \frac{4\Delta}{I_0^{3/2}\epsilon}(\sqrt{b}+\Exp_W[\sqrt{|W|}]) +C'
    \\
    &= \frac{b}{I_0} + \frac{4\Delta}{I_0^{3/2}\epsilon}\sqrt{b}+ C'',
\end{align*}
where $C'$ and $C''$ are both constants independent of $b$. 

\end{proof}

\begin{proof}[Proof of Theorem \ref{thm4}]
Similar to the proof in Theorem  \ref{thm:dp}, for $t\in\mathbb N$ and for any sequence $X_{(1:t)}$ and $X'_{(1:t)}$ that differs only at time $k$, we have $S_1=S'_1, \ldots, S_{k-1}=S'_{k-1},$ but $S_k\neq S'_k$. Without loss of generality, we assume $S_k > S'_k$. By the definition of CUSUM statistics \eqref{eq:1}, we then have $S_{k+1} \ge S'_{k+1}$, $\ldots$, $S_t \ge S'_t$. 
Since we fixed $Z_1,\dots,Z_{t-1}$, we have $\tilde S_1 = \tilde S'_1,\ldots, \tilde S_{k-1} = \tilde S'_{k-1}; \tilde S_k > \tilde S'_k, \tilde S_{k+1} \ge \tilde S'_{k+1}, \ldots, \tilde S_{t-1} \ge \tilde S'_{t-1}$. 

Denote the event 
\[
\mathcal{E}_\delta:=\{|\ell(X_k)|\le \frac{A_\delta}{2} \text{ and }|\ell(X'_k)|\le \frac{A_\delta}{2}\}.
\]
Note that conditioning on event $\mathcal{E}_\delta$, we have $|\tilde S_k - \tilde S'_k| \le A_\delta$ and thus $|\tilde S_j - \tilde S'_j|\le A_\delta$ for $j>k$. Therefore, we have $\max_{1\le i\le t-1}\tilde{S'_j} - \max_{1\le i\le t-1}\tilde{S_j}\le 0$ and $|\max_{1\le i\le t-1}\tilde{S'_j} - \max_{1\le i\le t-1}\tilde{S_j}|\le A_\delta$. Meanwhile, we have $S_t - S_t' \in [0,A_\delta]$. By proof of Theorem \ref{thm:dp}, we have $\Pro_{Z_t,W}(\tilde{T} =t \mid X_{(1:n)})\le e^\epsilon \Pro_{Z_t,W}(\tilde{T} =t \mid X'_{(1:t)} )$ still hold, conditioning on $\mathcal{E}_\delta$.

By the definition of $A_\delta$, we have for any given $k$, under both pre- and post-change measure,
\begin{align*}
\Pro( \mathcal{E}_\delta ) \ge 1-\delta, \text{ i.e., }  \Pro( \mathcal{E}_\delta^c ) \le \delta. 
\end{align*}
Combining these, we have
\[
\begin{aligned}
& \Pro(\tilde{T} =t \mid X_{(1:t)\setminus k}) \\
& = \Pro(\tilde{T} =t \mid X_{(1:t) \setminus k}, \mathcal{E}_\delta)\Pro(\mathcal{E}_\delta) \\
& \hspace{10pt} + \Pro(\tilde{T} =t \mid X'_{(1:t)\setminus k}, \mathcal{E}_\delta^c)\Pro(\mathcal{E}_\delta^c) \\
 & \leq e^\epsilon \Pro(\tilde{T} =t \mid X'_{(1:t)\setminus k},\mathcal{E}_\delta)\Pro(\mathcal{E}_\delta) + \delta \\
 & \le e^\epsilon \Pro(\tilde{T} =t \mid X'_{(1:t)\setminus k}) + \delta.
\end{aligned}
\]
This concludes the proof of Theorem \ref{thm4}.
\end{proof}

\begin{proof}[Proof of Corollary \ref{thm:edd2}]
We first note that the proof of Lemma~\ref{lem:wadd0} remains valid, and thus the worst-case detection delay of $\tilde{T}$ is still attained when the changepoint occurs at $\tau = 0$. 
Furthermore, the proof of Theorem~\ref{thm:edd} does not explicitly rely on the boundedness of the log-likelihood ratio $\ell(\cdot)$; it only invokes the boundedness of $\ell(X)$ through the use of Hoeffding’s inequality. 
As a result, the same proof applies in the unbounded case provided that $\ell(X)$ is sub-Gaussian under both the pre- and post-change distributions. Specifically, assume there exists a constant $\sigma^2 > 0$ such that $\Exp[e^{\lambda (\ell(X) - \Exp[\ell(X)]) }] \le e^{ \frac{\lambda^2\sigma^2}{2}}$ for all $\lambda \in\mathbb R$. 
Under this condition, the same argument in the proof of Theorem~\ref{thm:edd} continues to hold with Laplace noise $\text{Lap}(2\Delta/\epsilon)$ replaced by $\text{Lap}(2A_\delta/\epsilon)$. The sub-Gaussian constant $\sigma^2$ appears only in the additive constant term $C$ in the resulting WADD bound. We omit the technical derivations, as they are essentially the same as the proof of Theorem~\ref{thm:edd}.
\end{proof}

\bibliographystyle{IEEEtran}
\bibliography{ref}



\end{document}